\documentclass[reqno]{amsart}

\usepackage{amsfonts, upgreek}
\usepackage{amssymb, amsopn, latexsym,amsmath,graphics,verbatim,amsthm,enumerate,amscd,tabularx}
\usepackage[all,cmtip]{xy}
\usepackage{mathtools, graphicx}
\usepackage{fancyhdr}
\usepackage[active]{srcltx}
\usepackage[british]{babel}
\usepackage{amsmath,amssymb,amsthm,mathrsfs, url}

\theoremstyle{plain}
\newtheorem*{theorem*}{Theorem}
\newtheorem*{conjecture*}{Conjecture}
\newtheorem*{conjectureA}{Conjecture A}
\newtheorem*{conjectureA*}{Conjecture A*}
\newtheorem*{conjectureB}{Conjecture B}
\newtheorem*{conjectureC}{Conjecture C}

\newtheorem{thm}{Theorem}[section]
\newtheorem{theorem}[thm]{Theorem}

\newtheorem{corollary}[thm]{Corollary}
\newtheorem{lemma}[thm]{Lemma}
\newtheorem*{lemma*}{Lemma}

\newtheorem{proposition}[thm]{Proposition}
\theoremstyle{definition}

\theoremstyle{remark}

\input cyracc.def \font\tencyr=wncyr10 \def\russe{\tencyr\cyracc}
\def\Sha{\text{\russe{Sh}}}
\def\FSha{\text{\russe{Zh}}}

\newdir{ (}{{}*!/-5pt/@^{(}}

\SelectTips{eu}{10}
\UseTips
\entrymodifiers={+!!<0pt,\fontdimen22\textfont2>}

\DeclareMathOperator{\Gal}{Gal}
\DeclareMathOperator{\rank}{rank}
\DeclareMathOperator{\corank}{corank}

\DeclareMathOperator{\ord}{ord}
\DeclareMathOperator{\Selm}{Sel}

\DeclareMathOperator{\img}{img}
\DeclareMathOperator{\coker}{coker}
\DeclareMathOperator{\Hom}{Hom}

\DeclareMathOperator{\res}{res}

\DeclareMathOperator{\dual}{dual}

\DeclareMathOperator{\divs}{div}

\newcommand{\Q}{{\mathbb{Q}}}
\newcommand{\Z}{{\mathbb{Z}}}
\newcommand{\N}{{\mathbb{N}}}

\newcommand{\C}{{\mathbb{C}}}

\newcommand{\Fp}{{\mathbb{F}_p}}
\newcommand{\Zp}{{\mathbb{Z}_p}}
\newcommand{\Zpu}{{\mathbb{Z}^{\times}_p}}
\newcommand{\Zl}{{\mathbb{Z}_l}}

\newcommand{\Qp}{{\mathbb{Q}_p}}
\newcommand{\ilim}{\mathop{\varprojlim}\limits}
\newcommand{\dlim}{\mathop{\varinjlim}\limits}
\newcommand{\Sel}{{\Selm_p}}
\newcommand{\Selinf}{{\Selm_{p^{\infty}}}}

\newcommand{\cN}{\mathcal{N}}

\newcommand{\cO}{\mathcal{O}}

\newcommand{\cy}[1]{\mathbb{Z}/#1\mathbb{Z}}
\newcommand{\fp}{\mathfrak{p}}

\newcommand{\overbar}[1]{\mkern 1.5mu\overline{\mkern-1.5mu#1\mkern-1.5mu}\mkern 1.5mu}

\newcommand{\joinrelshort}{\mathrel{\mkern-9mu}}
\newcommand{\shortlongrightarrow}{\relbar\joinrelshort\rightarrow}
\newcommand{\isomarrow}{\mathrel{\mathop{\setbox0\hbox{$\mathsurround0pt
        \shortlongrightarrow$}\ht0=0.7\ht0\box0}\limits
        ^{\sim\mkern2mu}}}

\begin{document}

\title{Fine Selmer Groups, Heegner points and Anticyclotomic $\Zp$-extensions}

\author{Ahmed Matar}

\address{Department of Mathematics\\
         University of Bahrain\\
         P.O. Box 32038\\
         Sukhair, Bahrain}
\email{amatar@uob.edu.bh}

\begin{abstract}
Let $E/\Q$ be an elliptic curve, $p$ a prime and $K_{\infty}/K$ the anticyclotomic $\Zp$-extension of a quadratic imaginary field $K$ satisfying the Heegner hypothesis. In this paper we make a conjecture about the fine Selmer group over $K_{\infty}$. We also make a conjecture about the structure of the module of Heegner points in $E(K_{\fp_{\infty}})/p$ where $K_{\fp_{\infty}}$ is the union of the completions of the fields $K_n$ at a prime of $K_{\infty}$ above $p$. We prove that these conjectures are equivalent. When $E$ has supersingular reduction at $p$ we also show that these conjectures are equivalent to the conjecture in our earlier work. Assuming these conjectures when $E$ has supersingular reduction at $p$, we prove various results about the structure of the Selmer group over $K_{\infty}$.
\end{abstract}

\maketitle

\section{Introduction}

Let $K$ be an imaginary quadratic field with discriminant $d_K \neq -3,-4$ whose class number we will denote by $h_K$.

Let $p \geq 5$ be a prime and $E$ an elliptic curve of conductor $N$ defined over $\Q$ with a modular parametrization $\pi: X_0(N) \to E$. We shall say that $(E,p)$ satisfies $(*)$ if the following are met:
\begin{enumerate}[(i)]
\item All the primes dividing $N$ split in $K/\Q$
\item $p$ does not divide $Nd_Kh_K \varphi(Nd_K)$
\item $\Gal(\Q(E[p])/\Q)=GL_2(\Fp)$
\item If $E$ has ordinary reduction at $p$ then
\begin{enumerate}
\item $p \nmid \#E(\Fp)$
\item $a_p \not \equiv -1 \; (\text{mod } p)$ if $p$ is inert in $K/\Q$
\item $a_p \not \equiv 2 \; (\text{mod } p)$ if $p$ splits in $K/\Q$
\end{enumerate}
\end{enumerate}

We shall say that $(E, \pi, p)$ satisfies $(*)$ if $(E,p)$ satisfies $(*)$ and furthermore $p$ does not divide the number of geometrically connected components of the kernel of $\pi_*: J_0(N) \to E$.

Let $K_{\infty}/K$ be the anticyclotomic $\Zp$-extension of $K$, $\Gamma=\Gal(K_{\infty}/K)$ and $K_n$ the unique subfield of $K_{\infty}$ containing $K$ such that $\Gal(K_n/K) \cong \cy{p^n}$. Denote $\Gamma_n=\Gamma^{p^n}$, $G_n=\Gamma/\Gamma_n$ and $R_n=\Fp[G_n]$.

Let $\Lambda=\Zp[[\Gamma]]$ be the Iwasawa algebra attached to $K_{\infty}/K$. Fixing a topological generator $\gamma \in \Gamma$ allows us to identify $\Lambda$ with the power series ring $\Zp[[T]]$. Also consider the ``mod $p$'' Iwasawa algebra $\overbar{\Lambda}=\Lambda/p\Lambda=\Fp[[T]]$.

Now let $E'$ be a strong Weil curve in the isogeny class of $E$ i.e. there exists a modular parametrization $\pi': X_0(N) \to E'$ which maps the cusp $\infty$ of $X_0(N)$ to the origin of $E'$ such that the induced map $\pi'_*: J_0(N) \to E'$ has a geometrically connected kernel.

If we assume that all the primes dividing $N$ split in $K/\Q$, then choosing an ideal $\cN$ of $\cO_K$ such that $\cO_K/\cN \cong \cy{N}$ allows us to define a family of Heegner points $\alpha_n \in E(K_n)$ using the modular parametrization $\pi$ and a family of Heegner points $\alpha'_n \in E'(K_n)$ using the modular parametrization $\pi'$ (see section 2). In \cite{Matar} we made the following conjecture

\begin{conjectureA*}
Assume that $(E,p)$ satisfies $(*)$, $p$ splits in $K/\Q$ and $E$ has supersingular reduction at $p$ then the $\overbar{\Lambda}$-submodule of $E'(K_{\infty})/p$ generated by the Heegner points $\alpha'_n$ has $\overbar{\Lambda}$-corank greater than or equal to two.
\end{conjectureA*}

Assuming this conjecture, we proved in \cite{Matar} (using the same notation in that paper) that if $E$ has supersingular reduction at $p$ and $p$ splits in $K/\Q$ then the $\Lambda$-corank of $\Selinf(E/K_{\infty})$ is equal to 2 and that $X_{p^{\infty}}(E/K_{\infty})=\{0\}$.

We now make the slightly stronger conjecture

\begin{conjectureA}
Assume that $(E, \pi, p)$ satisfies $(*)$, $p$ splits in $K/\Q$ and $E$ has supersingular reduction at $p$ then the $\Gamma$-submodule of $E(K_{\infty})/p$ generated by the Heegner points $\alpha_n$ has $\overbar{\Lambda}$-corank greater than or equal to two.
\end{conjectureA}

It is easy to see that conjecture A implies conjecture A*. We proved theorem B of \cite{Matar} for a strong Weil curve $E'$ isogenous to $E$. As theorem B was invariant under isogeny, conjecture A* sufficed for our purposes. The author has not been able to prove that the results in this paper are invariant under isogeny which is the reason we have chosen to work with the slightly stronger conjecture A.

We now define the fine $l^{\infty}$-Selmer group. Assume that $l$ is an odd prime, $F$ a number field and $\mathcal{E}$ is an elliptic curve defined over $F$. Let $S$ be a finite set of primes of $F$ containing all the primes dividing $l$ and all the primes where $\mathcal{E}$ has bad reduction. We let $F_S$ be the maximal extension of $F$ unramified outside $S$. Suppose now that $L$ is a field with $F \subseteq L \subseteq F_S$. We let $G_S(L)=\Gal(F_S/L)$ and $S_L$ be the set of primes of $L$ above those in $S$. We define the fine $l^{\infty}$-Selmer group of $\mathcal{E}/L$ as

$$\displaystyle 0 \longrightarrow R_{l^{\infty}}(\mathcal{E}/L) \longrightarrow H^1(G_S(L), \mathcal{E}[l^{\infty}]) \longrightarrow \prod_{v \in S_L} H^1(L_v, \mathcal{E}[l^{\infty}])$$

If $l$ is a fixed prime then for any number field $F$ we let $F^{cyc}$ denote the cyclotomic $\Zl$-extension of $F$. In \cite{CS} Coates and Sujatha make the following conjecture

\begin{conjecture*}[Coates-Sujatha]
If $l$ is an odd prime, $F$ a number field and $\mathcal{E}$ an elliptic curve defined over $F$, then $R_{l^{\infty}}(\mathcal{E}/F^{cyc})$ is a cofinitely generated $\Zl$-module.
\end{conjecture*}

If $l$ is a fixed prime and $F$ is an imaginary quadratic field we let $F^{anti}$ denote the anticyclotomic $\Zl$-extension of $F$ (we have denoted this by $K_{\infty}$ for the imaginary quadratic field $K$ above). In relation to the conjecture of Coates and Sujatha above we propose the following conjecture

\begin{conjectureB}

If $l$ is an odd prime, $F$ an imaginary quadratic field and $\mathcal{E}$ an elliptic curve defined over $F$, then $R_{l^{\infty}}(\mathcal{E}/F^{anti})$ is a cofinitely generated $\Zl$-module.
\end{conjectureB}

Our first result in this paper is the following theorem

\begin{theorem*}
Assume that $(E, \pi, p)$ satisfies $(*)$, $p$ splits in $K/\Q$ and $E$ has supersingular reduction at $p$, then conjecture A and conjecture B (for $\mathcal{E}=E$, $F=K$ and $l=p$) are equivalent.
\end{theorem*}

If $\fp_{\infty}$ is a prime of $K_{\infty}$ above $p$, we let $K_{\fp_{\infty}}$ denote the union of the completions of the fields $K_1 \subset K_2 \subset K_3 \subset \cdots$ at $\fp_{\infty}$. We now propose a third conjecture

\begin{conjectureC}
Suppose that $(E, \pi, p)$ satisfies $(*)$
\begin{enumerate}[(i)]
\item If $E$ has ordinary reduction at $p$, then there exists a prime $\fp_{\infty}$ of $K_{\infty}$ above $p$ such that the $\Gamma$-submodule of $E(K_{\fp_{\infty}})/p$ generated by the Heegner points $\alpha_n$ has infinite cardinality.
\item If $E$ has supersingular reduction at $p$ and $p$ splits in $K/\Q$, then there exists a prime $\fp_{\infty}$ of $K_{\infty}$ above $p$ such that the $\Gamma$-submodule of $E(K_{\fp_{\infty}})/p$ generated by the Heegner points $\alpha_{2n}$ and the $\Gamma$-submodule of $E(K_{\fp_{\infty}})/p$ generated by the Heegner points $\alpha_{2n+1}$ are both of infinite cardinalities.
\end{enumerate}
\end{conjectureC}

If one replaces $K_{\fp_{\infty}}$ in the above conjecture with the global field $K_{\infty}$ then the conjecture becomes true. This can be shown using the results of Cornut \cite{Cornut} and Cornut and Vatsal \cite{CV_DOC} (see theorems 3.1 and 4.1 of \cite{Matar}). Therefore, the conjecture is a local analog of this global result.

The relationship between conjectures B and C is given in the following theorem

\begin{theorem*}
Assume that $(E, \pi, p)$ satisfies $(*)$ then we have
\begin{enumerate}[(a)]
\item If $E$ has ordinary reduction at $p$, then conjecture C(i) and conjecture B (for $\mathcal{E}=E$, $F=K$ and $l=p$) are equivalent.
\item If $p$ splits in $K/\Q$ and $E$ has supersingular reduction at $p$, then conjecture C(ii) and conjecture B (for $\mathcal{E}=E$, $F=K$ and $l=p$) are equivalent (and hence also equivalent to conjecture A by the previous theorem).
\end{enumerate}
\end{theorem*}

The final result of this paper concerns the $\mu$-invariant of the Pontryagin dual of $\Selinf(E/K_{\infty})$ which we denote by $\Selinf(E/K_{\infty})^{\dual}$ (see section 2 for a definition of $\Selinf(E/K_{\infty})$). Using the method of proof of \cite{Matar} we will show

\begin{theorem*}
Suppose that $(E, \pi, p)$ satisfies $(*)$ then we have
\begin{enumerate}[(a)]
\item If $E$ has ordinary reduction at $p$, then $\Selinf(E/K_{\infty})^{\dual}$ has $\Lambda$-rank equal to 1 and $\mu$-invariant equal to zero.
\item If $p$ splits in $K/\Q$, $E$ has supersingular reduction at $p$ and conjecture C(ii) is true, then $\Selinf(E/K_{\infty})^{\dual}$ has $\Lambda$-rank equal to 2 and $\mu$-invariant equal to zero.
\end{enumerate}
\end{theorem*}

Using the results of Wuthrich \cite{Wuthrich1}, we end this paper in section 4 by giving examples veryfiying conjecture B.

\section{Definitions and Control Theorems}
\subsection{Definitions}
In this section we recall the definition of the Heegner points as well as the definition of the Selmer and fine Selmer group.

Let $E$ be an elliptic curve defined over $\Q$. We fix a modular parametrization $\pi: X_0(N) \to E$ which maps the cusp $\infty$ of $X_0(N)$ to the origin of $E$ (see \cite{Wiles} and \cite{BDCT}). Assume that every prime dividing $N$ splits in $K/\Q$ (condition $(*)$-i). It follows that we can choose an ideal $\cN$ such that $\cO_K/\cN \cong \Z/N\Z$. Let $m$ be an integer that is relatively prime to $Nd_K$ and let $\cO_m = \Z + m\cO_K$ be the order of conductor $m$ in $K$. The ideal $\cN_m=\cN \cap \cO_m$ satisfies $\cO_m/\cN_m \cong \Z/N\Z$ and therefore the natural projection of complex tori:

$$\C/\cO_m \to \C/\cN_m^{-1}$$\\
is a cyclic $N$-isogeny, which corresponds to a point of $X_0(N)$. Let $\alpha[m]$ be its image under the modular parametrization $\pi$. From the theory of complex multiplication we have that $\alpha[m] \in E(K[m])$ where $K[m]$ is the ring class field of $K$ of conductor $m$.

If we assume that the class number of $K$ is not divisible by $p$ (condition $(*)$-ii), it follows for any $n$ that $K[p^{n+1}]$ is the ring class field of minimal conductor that contains $K_n$. We now define $\alpha_n \in E(K_n)$ to be the trace from $K[p^{n+1}]$ to $K_n$ of $\alpha[p^{n+1}]$.

Let $R_n\alpha_n$ denote the $R_n$-submodule of $H^1(K_n, E[p])$ generated by the image of $\alpha_n$ under the Kummer map

$$E(K_n) \to H^1(K_n, E[p]).$$\\

If we assume that $\Gal(\Q(E[p])/\Q)=GL_2(\Fp)$ (condition $(*)$-iii), then by corollary 2.4 of \cite{Matar} we have $E(K_{\infty})[p^{\infty}]=\{0\}$. This implies that the restriction map for $m\geq n$

$$H^1(K_n, E[p]) \to H^1(K_m, E[p])$$\\
\noindent is injective and therefore allows us to view $R_n\alpha_n$ as a submodule of $H^1(K_m, E[p])$.

Assume that $E$ has good ordinary reduction at $p$ and condition $(*)$-iv is met. Then we have (see \cite{Bertolini1} prop 2.1.4) $\text{Tr}_{K_{n+1}/K_n}(\alpha_{n+1})=u\alpha_n$ for some unit $u \in R_n$. This implies that $R_n\alpha_n \subset R_{n+1}\alpha_{n+1}$ and so we may construct the direct limit $\dlim R_n \alpha_n$.

Now assume that $E$ has good supersingular reduction at $p \geq 5$. In this case one can show that $\text{Tr}_{K_{n+1}/K_n}(\alpha_{n+1})=-\alpha_{n-1}$. This then implies that $R_n\alpha_n \subset R_{n+2}\alpha_{n+2}$ and so we may construct the direct limits $\dlim R_{2n} \alpha_{2n}$ and $\dlim R_{2n+1} \alpha_{2n+1}$.

Let us now define the Selmer groups we will be working with: If $L/\Q$ is any algebraic extension and $E$ is an elliptic curve defined over $L$, we let $\Selinf(E/L)$ denote the $p^{\infty}$-Selmer group of $E$ over $L$ defined by
$$\displaystyle 0 \longrightarrow \Selinf(E/L) \longrightarrow H^1(L, E[p^{\infty}])\longrightarrow \prod_v H^1(L_v, E)[p^{\infty}].$$

We will also be working with the $p$-Selmer group $\Sel(E/L)$ defined by

$$\displaystyle 0 \longrightarrow \Sel(E/L) \longrightarrow H^1(L, E[p])\longrightarrow \prod_v H^1(L_v, E)[p].$$

We now repeat the definition of the fine $p^{\infty}$-Selmer group from the introduction. Assume that $p$ is an odd prime, $F$ a number field and $E$ is a an elliptic curve defined over $F$. Let $S$ be a finite set of primes of $F$ containing all the primes dividing $p$ and all the primes where $E$ has bad reduction. We let $F_S$ be the maximal extension of $F$ unramified outside $S$. Suppose now that $L$ is a field with $F \subseteq L \subseteq F_S$. We let $G_S(L)=\Gal(F_S/L)$ and $S_L$ be the set of primes of $L$ above those in $S$. We define the fine $p^{\infty}$-Selmer group of $E/L$ as

$$\displaystyle 0 \longrightarrow R_{p^{\infty}}(E/L) \longrightarrow H^1(G_S(L), E[p^{\infty}]) \longrightarrow \prod_{v \in S_L} H^1(L_v, E[p^{\infty}]).$$

The definition of $R_{p^{\infty}}(E/L)$ does not depend on the set $S$. In fact, one can show that for any set $S$ as above we have

$$\displaystyle 0 \longrightarrow R_{p^{\infty}}(E/L) \longrightarrow \Selinf(E/L) \longrightarrow \prod_{v | p} H^1(L_v, E[p^{\infty}]).$$

We also define the fine $p$-Selmer group of $E/L$ whose definition may depend on the set $S$. It is defined as

$$\displaystyle 0 \longrightarrow R^S_p(E/L) \longrightarrow H^1(G_S(L), E[p]) \longrightarrow \prod_{v \in S_L} H^1(L_v, E[p]).$$

\subsection{Control Theorems}
In this section we prove two Iwasawa-theoretic control theorems: one for the $p$-Selmer group and another for the fine $p$-Selmer group

First we need the following proposition

\begin{proposition}\label{Iwasawa_rank_proposition}
Let $M$ be a finitely generated $\overbar{\Lambda}$-module. Consider the $\overbar{\Lambda}$-module $M^+=\Hom_{\overbar{\Lambda}}(M, \overbar{\Lambda})$. Then $M^+$ is a free $\overbar{\Lambda}$-module with the same rank as $M$ and we have an isomorphism $M^+\cong \ilim_n (M^{\; \dual})^{\Gamma_n}$ where $M^{\; \dual}=\Hom(M, \Fp)$ is the Pontryagin dual of $M$ and the inverse limit is with respect to the norm maps.
\end{proposition}

\begin{proof}

The proof is basically the same as \cite{PR} 2.2 lemma 4. The fact that $M^+$ is a free $\overbar{\Lambda}$-module with the same rank as $M$ is clear. As for the second statement we have the following isomorphisms

\begin{align*}
M^+&=\Hom_{\overbar{\Lambda}}(M, \overbar{\Lambda})\\
&\cong \ilim_n \Hom_{\overbar{\Lambda}}(M, \Fp[G_n])\\
&\cong \ilim_n \Hom_{\Fp[G_n]}(M_{\Gamma_n}, \Fp[G_n])\\
&\cong \ilim_n \Hom_{\Fp}(M_{\Gamma_n}, \Fp)= \ilim_n (M^{\; \dual})^{\Gamma_n}.
\end{align*}

The last isomorphism above is induced by the isomorphism\\ $\Hom_{\Fp}(M_{\Gamma_n}, \Fp) \cong \Hom_{\Fp[G_n]}(M_{\Gamma_n}, \Fp[G_n])$: $\varphi \mapsto (x \mapsto \sum_{g \in G_n} \varphi(g^{-1}x)g)$.

\end{proof}

We also need the following lemma

\begin{lemma}\label{inverse_limit_lemma}
If $X=\ilim M_i$ is the inverse limit of finite groups of bounded order, then $X$ is finite.
\end{lemma}
\begin{proof}
Giving the groups $M_i$ the discrete topology makes $X$ a profinite group. Since the groups $M_i$ have bounded order, therefore it follows from \cite{Wilson} prop 1.1.6(b) that all open subgroups of $X$ have bounded index. Since every nontrivial profinite group contains a proper open subgroup and since open subgroups of profinite groups are themselves a profinite group therefore the indexes of the open subgroups in an infinite profinite group must be unbounded. It follows that $X$ must be finite.
\end{proof}

Assume that $(E, \pi, p)$ satisfies $(*)$ and $S$ is a finite set of primes of $K$ containing all the primes dividing $p$ and all the primes where $E$ has bad reduction. We now define $X^S_{f,p}(E/K_{\infty}):=\ilim R^S_p(E/K_n)$ where the inverse limit is taken over $n$ with respect to the corestriction maps. Also
define $Y^S_{f,p}(E/K_{\infty}):=\ilim R^S_p(E/K_{\infty})^{\Gamma_n}$ where the inverse limit is taken over $n$ with respect to the norm maps.

The restriction maps $\res: R^S_p(E/K_n) \to R^S_p(E/K_{\infty})^{\Gamma_n}$ induce a map

$$\Xi: X^S_{f,p}(E/K_{\infty}) \to Y^S_{f,p}(E/K_{\infty}).$$\\

We now have the following control theorem

\begin{theorem}\label{control_theorem1}
Assume that $(E, \pi, p)$ satisfies $(*)$. If $S$ is the set of primes of $K$ dividing $Np$ then the map $\Xi: X^S_{f,p}(E/K_{\infty}) \to Y^S_{f,p}(E/K_{\infty})$ is an injection with finite cokernel.
\end{theorem}

\begin{proof}
Let $S_n$ be all the primes of $K_n$ dividing $S$ and $S_{\infty}$ all the primes of $K_{\infty}$ dividing $S$ and consider the following commutative diagram
\begin{equation}
\xymatrix{
0 \ar[r] & R^S_p(E/K_{\infty})^{\Gamma_n} \ar[r] & H^1(K_{\infty}, E[p])^{\Gamma_n} \ar[r]^-{\psi_{\infty}} & \underset{v \in S_{\infty}} {\bigoplus} H^1(K_{\infty,v}, E[p])^{\Gamma_n} \\
0 \ar[r] & R^S_p(E/K_n) \ar[u]_{s_n} \ar[r] & H^1(K_n, E[p]) \ar[u]_{h_n} \ar[r]^-{\psi_n}
& \underset{v \in S_n} {\bigoplus} H^1(K_{n,v}, E[p]) \ar[u]_{g_n}
}
\end{equation}

Taking the inverse limit of the groups in the top row with respect to the norm and the groups in the bottom row with respect to corestriction, we obtain the following diagram

\begin{equation}
\xymatrix{
0 \ar[r] & Y^S_{f,p}(E/K_{\infty}) \ar[r] & \ilim H^1(K_{\infty}, E[p])^{\Gamma_n} \ar[r]^-{\upphi}
& \ilim \underset{v \in S_{\infty}} \bigoplus H^1(K_{\infty,v}, E[p])^{\Gamma_n} \\
0 \ar[r] & X^S_{f,p}(E/K_{\infty}) \ar[u]_{\Xi} \ar[r] & \ilim H^1(K_n, E[p]) \ar[u]_{\Xi'} \ar[r]^-{\uppsi}
& \ilim \underset{v \in S_n} \bigoplus H^1(K_{n,v}, E[p]) \ar[u]_{\Xi''}
}
\end{equation}

By \cite{Matar} corollary 2.4 we have that $E(K_{\infty})[p^{\infty}]=\{0\}$ and so both $H^1(\Gamma_n, E(K_{\infty})[p^{\infty}])$ and $H^2(\Gamma_n, E(K_{\infty})[p^{\infty}])$ are trivial. This implies that the maps $h_n$ in the diagram (1) above are isomorphisms which in turn implies that the map $\Xi'$ in the diagram (2) is an isomorphism. Therefore by applying the snake lemma to the diagram (2) we see that $\Xi$ is an injection and $\coker \Xi =\img \uppsi \cap \ker \Xi''$ so the proof will be complete if we can show that $\ker \Xi''$ is finite.

Since we have assumed that all the primes dividing $N$ split in $K/\Q$, therefore it follows from \cite{Brink} th. 2 that the set $S_{\infty}$ is finite.

Now choose an $M$ such that $\#S_M=\#S_{\infty}$ and such that for every $w \in S_{\infty}$ we have $E(K_{\infty,w})[p]=E(K_{M,v})[p]$ where $v$ is the prime of $S_M$ below $w$.

Let $m=\#S_M$. For every $n \geq M$ we label the primes in $S_n$ as $v_1, v_2, ..., v_m$ and the primes of $S_{\infty}$ as $w_1, w_2, ..., w_m$. We choose a labelling such that if $k \geq j \geq M$ then $w_i \in S_{\infty}$ lies above $v_i \in S_k$ lies above $v_i \in S_j$. With this labelling we have

$$\ker \Xi''= \bigoplus_{i=1}^m \ilim_{n \geq M} H^1(\Gal(K_{\infty, w_i}/K_{n, v_i}), E(K_{\infty, w_i})[p])$$

where the inverse limit is taken over $n$ with respect to the corestriction maps.

For any $n \geq M$ and any $i$ we have $\Gal(K_{\infty, w_i}/K_{n, v_i})=\Gamma_n$, therefore if $g$ is a topological generator of $\Gamma$ we have $H^1(\Gal(K_{\infty, w_i}/K_{n, v_i}), E(K_{\infty, w_i})[p])=E(K_{\infty, w_i})[p]/(g^{p^n}-1)E(K_{\infty, w_i})[p]$ but $E(K_{\infty, w_i})[p]=E(K_{n, v_i})[p]$ so $(g^{p^n}-1)E(K_{\infty, w_i})[p]=\{0\}$ i.e. $H^1(\Gal(K_{\infty, w_i}/K_{n, v_i}), E(K_{\infty, w_i})[p])=E(K_{\infty, w_i})[p]$. For $n' \geq n \geq M$ one can check that the corestriction map from $H^1(\Gal(K_{\infty, w_i}/K_{n', v_i}), E(K_{\infty, w_i})[p])$ to $H^1(\Gal(K_{\infty, w_i}/K_{n, v_i}), E(K_{\infty, w_i})[p])$ is the identity map on $E(K_{\infty, w_i})[p]$ hence $\ilim_{n \geq M} H^1(\Gal(K_{\infty, w_i}/K_{n, v_i}), E(K_{\infty, w_i})[p])=E(K_{\infty, w_i})[p]$. This shows that $\ker \Xi''$ is finite which completes the proof.
\end{proof}

\begin{corollary}\label{control_theorem_corollary1}
Assume that $(E, \pi, p)$ satisfies $(*)$. If $S$ is the set of primes of $K$ dividing $Np$ then $R^S_p(E/K_{\infty})^{\dual}$ and $X^S_{f,p}(E/K_{\infty})$ are both finitely generated $\overbar{\Lambda}$-modules and $\corank_{\overbar{\Lambda}}(R^S_p(E/K_{\infty}))=\rank_{\overbar{\Lambda}}(X^S_{f,p}(E/K_{\infty}))$.
\end{corollary}

\begin{proof}
By \cite{Manin} th. 4.5 we know that $\Selinf(E/K_{\infty})^{\dual}$ is a finitely generated $\Lambda$-module. Since $E(K_{\infty})[p^{\infty}]=\{0\}$ by \cite{Matar} corollary 2.4, therefore we have an isomorphism $\Sel(E/K_{\infty}) \isomarrow \Selinf(E/K_{\infty})[p]$ and so $\Sel(E/K_{\infty})^{\dual}$ is a finitely generated $\overbar{\Lambda}$-module. The same then holds for $R^S_p(E/K_{\infty})^{\dual}$ since $R^S_p(E/K_{\infty}) \subseteq \Sel(E/K_{\infty})$ so by proposition \ref{Iwasawa_rank_proposition} $Y^S_{f,p}(E/K_{\infty})$ is also a finitely generated $\overbar{\Lambda}$-module. Since by the control theorem we have an injection $X^S_{f,p}(E/K_{\infty}) \hookrightarrow Y^S_{f,p}(E/K_{\infty})$, therefore we also have that $X^S_{f,p}(E/K_{\infty})$ is a finitely generated $\overbar{\Lambda}$-module. The corollary now follows from the control theorem and proposition \ref{Iwasawa_rank_proposition}.
\end{proof}

We now define $X_{s,p}(E/K_{\infty}):=\ilim \Sel(E/K_n)$ where the inverse limit is taken over $n$ with repect to the corestriction maps. Note that we have chosen to put an ``s" in the subscript so that the reader does not confuse this group with the group $X_p(E/K_{\infty})$ in \cite{Matar} which was defined in a different way.

We also define $Y_{s,p}(E/K_{\infty})=\ilim \Sel(E/K_{\infty})^{\Gamma_n}$ where the inverse limit is taken over $n$ with respect to the norm maps.

The restriction maps $\res: \Sel(E/K_n) \to \Sel(E/K_{\infty})^{\Gamma_n}$ induce a map

$$\Xi: X_{s,p}(E/K_{\infty}) \to Y_{s,p}(E/K_{\infty}).$$\\
We will now prove an Iwasawa-theoretic control theorem for the $p$-Selmer group. The theorem can be thought of as a ``mod $p$" analog of theorem 2.10 in \cite{Matar}.

\begin{theorem}\label{control_theorem2}
Suppose that $(E, \pi, p)$ satisfies $(*)$. Consider the map $\Xi$ induced by restriction

$$\Xi: X_{s,p}(E/K_{\infty}) \to Y_{s,p}(E/K_{\infty}).$$\\
\begin{enumerate}[(a)]
\item If $E$ has ordinary reduction at $p$, then $\Xi$ is an injection with finite cokernel.
\item If $E$ has supersingular reduction at $p$ and $p$ splits in $K/\Q$ and conjecture C(ii) is true, then $\Xi$ is an injection and $\rank_{\overbar{\Lambda}}(\coker \Xi) \le 2$.
\end{enumerate}

\end{theorem}
\begin{proof}
First we prove part (a): Assume that $E$ has ordinary reduction at $p$. From Mazur's control theorem (\cite{Mazur}; see also \cite{Gb_LNM} and \cite{Gb_IIC}) using the fact that $E(K_{\infty})[p^{\infty}]=\{0\}$ (\cite{Matar} corollary 2.4) we get that for any $n$ the restriction map

$$\res_n: \Selinf(E/K_n) \to \Selinf(E/K_{\infty})^{\Gamma_n}$$\\
is an injection with finite cokernel of bounded order as $n$ varies. Since $E(K_{\infty})[p^{\infty}]=\{0\}$, therefore for any $n$ we have an isomorphism $\Sel(E/K_n) \isomarrow \Selinf(E/K_n)[p]$ and an isomorphism $\Sel(E/K_{\infty}) \isomarrow \Selinf(E/K_{\infty})[p]$. Therefore for any $n$ the restriction map gives an exact sequence

$$\res_n: 0 \longrightarrow \Sel(E/K_n) \longrightarrow \Sel(E/K_{\infty})^{\Gamma_n} \longrightarrow C_n \longrightarrow 0$$\\
where $C_n$ is finite and of bounded order as $n$ varies. Part (a) then follows from this by taking inverse limits and using lemma \ref{inverse_limit_lemma}.

Now we prove part (b). The proof of this part is very similar to the proof of theorem 2.10(b) in \cite{Matar}. Assume that $E$ has supersingular reduction at $p$, $p$ splits in $K/\Q$ and conjecture C(ii) is true. Define $S=\{p\} \cup \{l \; \text{prime}: l|N\}$. For any $n$, with this set $S$, we define $S_n$ to be the set of primes of $K_n$ above those in $S$ and $S_{\infty}$ to be the primes of $K_{\infty}$ above those in $S$. Now define $K_S$ to be the maximal extension of $K$ unramified outside $S$, $G_S(K_n)=\Gal(K_S/K_n)$ and $G_S(K_{\infty})=\Gal(K_S/K_{\infty})$. Note that since we have assumed all the primes dividing $N$ to split in $K/\Q$, therefore it follows from theorem 2 of \cite{Brink} that the set $S_{\infty}$ is finite.

For any $K_n$ it is well-known that the $p$-Selmer group $\Sel(E/K_n)$ may be defined as

$$\displaystyle 0 \longrightarrow \Sel(E/K_n) \longrightarrow H^1(G_S(K_n), E[p])\longrightarrow \prod_{v\in S_n} H^1(K_{n,v} E)[p].$$

We may also define $\Sel(E/K_{\infty})$ as

$$\displaystyle 0 \longrightarrow \Sel(E/K_{\infty}) \longrightarrow H^1(G_S(K_{\infty}), E[p])\longrightarrow \prod_{v\in S_{\infty}} H^1(K_{\infty,v} E)[p].$$

For any $n$ consider the following commutative diagram:

\begin{equation}
\xymatrix{
0 \ar[r] & \Sel(E/K_{\infty})^{\Gamma_n} \ar[r] & H^1(G_S(K_{\infty}), E[p])^{\Gamma_n} \ar[r]^-{\psi_{\infty}}
& \underset{v \in S_{\infty}}{\bigoplus} H^1(K_{\infty,v}, E)[p]^{\Gamma_n} \\
0 \ar[r] & \Sel(E/K_n) \ar[u]_{s_n} \ar[r] & H^1(G_S(K_n), E[p]) \ar[u]_{h_n} \ar[r]^-{\psi_n}
& \underset{v \in S_n} {\bigoplus} H^1(K_{n,v}, E)[p] \ar[u]_{g_n}
}
\end{equation}

The vertical maps in the above diagram are restriction. Let us note a few things related to this diagram:\\

(1) The maps $h_n$ are isomorphisms: This follows from the fact that $H^1(\Gamma_n, E(K_{\infty})[p^m])$ and $H^2(\Gamma_n, E(K_{\infty})[p^m])$ are both trivial because $E(K_{\infty})[p^{\infty}]=\{0\}$ (\cite{Matar} corollary 2.4).\\

(2) For any $v \in S_{\infty}$ above $p$ we have $H^1(K_{{\infty},v}, E)[p]=\{0\}$: The result follows from \cite{CG} cor. 3.2 as explained in \cite{Gb_LNM} pg. 70. Note that the fact that $E$ has supersingular reduction at $p$ is crucial for this result.\\

(3) For any $v \in S_n$ not dividing $p$ we have that $H^1(K_{n,v}, E)[p]$ is finite and of bounded order as $n$ varies: This follows from 2 facts. First, by Tate duality for abelian varieties over local fields (\cite{Milne} cor. 3.4) we have that $H^1(K_{n,v}, E)[p]$ is isomorphic to the dual of $E(K_{n,v})/p$. Secondly, if $l$ is the rational prime lying below $v$, then by Mattuck's theorems we have that $E(K_{n,v}) \cong \Z_l^d \times T$ where $d=[K_{n,v}:\Q_l]$ and $T$ is a finite group. Therefore it follows from these 2 facts that $\#H^1(K_{n,v}, E)[p]\le p^2$.\\

(4) Let $\fp_1$ and $\fp_2$ be the primes of $K$ above $p$. Since we have assumed that the class number of $K$ is relatively prime to $p$, therefore both $\fp_1$ and $\fp_2$ are totally ramified in $K_{\infty}/K$. So in particular there are only 2 primes $\fp_{n,1}$ and $\fp_{n,2}$ of $K_n$ above $p$ and 2 primes $\fp_{\infty,1}$ and $\fp_{\infty,2}$ of $K_{\infty}$ above $p$.\\

Let $\tilde{S}_{\infty}=S_{\infty} \backslash \{\fp_{\infty,1}, \fp_{\infty,2}\}$ (see (4)). Taking the points (2)-(4) into consideration, we take the inverse limit of the objects in the diagram above over $n$ (using the corestriction map for the bottom row and the norm map for the top row) to obtain the following diagram where the group $T$ is finite (using point (3) and lemma \ref{inverse_limit_lemma})

\begin{equation}
\xymatrix{
0 \ar[r] & Y_{s,p}(E/K_{\infty}) \ar[r] & \ilim H^1(G_S(K_{\infty}), E[p])^{\Gamma_n} \ar[r]^-{\upphi}
& \underset{v \in \tilde{S}_{\infty}}{\bigoplus} \ilim H^1(K_{\infty,v}, E)[p]^{\Gamma_n} \\
0 \ar[r] & X_{s,p}(E/K_{\infty}) \ar[u]_{\Xi} \ar[r] & \ilim H^1(G_S(K_n), E[p]) \ar[u]_{\Xi'} \ar[r]^-{\uppsi}
& T \times \underset{i=1,2}{\bigoplus} \ilim H^1(K_{\fp_{n,i}}, E)[p] \ar[u]_{\Xi''}
}
\end{equation}

To ease the notation, in the above diagram we have denoted $K_{n,{\fp_{n,i}}}$ by $K_{\fp_{n,i}}$. Applying the snake lemma to this diagram we get

$$0 \to \ker \Xi \to \ker \Xi' \to \ker \Xi'' \cap \img \uppsi \to \coker \Xi \to \coker \Xi'$$\\
From point (1) above, it follows that $\Xi'$ is an isomorphism i.e. $\ker \Xi'=0$ and $\coker \Xi'=0$. Therefore from the above sequence we get that $\ker \Xi=0$ as required. We also get that $\coker \Xi= \ker \Xi'' \cap \, \img \uppsi$. Since $T$ is finite and $\Xi''$ restricted to $H^1(K_{\fp_{n,i}}, E)[p]$ is the zero map, it follows that $\rank_{\overbar{\Lambda}}(\coker \Xi) = \rank_{\overbar{\Lambda}}(\img \uppsi)$.
Therefore we must show that $\rank_{\overbar{\Lambda}}(\img \uppsi) \le 2$. To study $\img \uppsi$ we use the Cassels-Poitou-Tate exact sequence (see \cite{CS}) which gives that the following sequence is exact

$$H^1(G_S(K_n), E[p]) \xrightarrow{\psi_n} \underset{v \in S_n}{\bigoplus} H^1(K_{n,v}, E)[p] \xrightarrow{\theta_n} \Sel(E/K_n)^{\dual}$$

We take the inverse limits of the groups over $n$ using the corestriction map. As all the groups we are dealing with are compact Hausdorff, the resulting sequence is also exact:

$$\ilim H^1(G_S(K_n), E[p]) \xrightarrow{\uppsi} T \times \underset{i=1,2}{\bigoplus} \ilim H^1(K_{\fp_{n,i}}, E)[p] \xrightarrow{\uptheta} \Sel(E/K_{\infty})^{\dual}$$

The fact that this sequence is exact means that $\img \uppsi = \ker \uptheta$. So to show that $\rank_{\overbar{\Lambda}}(\img \uppsi) \le 2$ it suffices to show that $\rank_{\overbar{\Lambda}}(\ker \uptheta) \le 2$ or equivalently, if $\hat{\uptheta}$ is the dual map, that $\corank_{\overbar{\Lambda}}(\coker \hat{\uptheta}) \le 2$.

By Tate local duality the dual of $H^1(K_{\fp_{n,i}}, E)[p]$ may be identified with $E(K_{\fp_{n,i}})/p$. Therefore using this fact, the map $\hat{\uptheta}$ becomes

$$\hat{\uptheta}: \Sel(E/K_{\infty}) \to E(K_{\fp_{\infty, 1}}) \otimes \Fp \times E(K_{\fp_{\infty, 2}}) \otimes \Fp.$$\\
This map is the usual map induced by restriction

$$H^1(K_{\infty}, E[p]) \to \underset{i=1,2}{\bigoplus} H^1(K_{\fp_{\infty,i}}, E[p]).$$\\
Note that if $c \in \Sel(E/K_{\infty}) \subset H^1(K_{\infty}, E[p])$ then its image under this map belongs to $E(K_{\fp_{\infty, 1}}) \otimes \Fp \times E(K_{\fp_{\infty,2}}) \otimes \Fp.$

To prove our result we will first calculate $\corank_{\overbar{\Lambda}}(E(K_{\fp_{\infty,i}}) \otimes \Fp)$. First we show that $E(K_{\fp_{\infty,i}})[p^{\infty}]=\{0\}$. Since $\Gamma=\Gal(K_{\fp_{\infty,i}}/\Qp)$ is pro-$p$, it suffices to show that $E(\Qp)[p^{\infty}]=E(K_{\fp_{\infty,i}})[p^{\infty}]^{\Gamma}=\{0\}$. But since $E$ has supersingular reduction at $p$, we have $E(\Qp)[p^{\infty}]=\hat{E}(p\Zp)[p^{\infty}]$ where $\hat{E}$ is the formal group of $E/\Qp$. The result then follows from the fact (\cite{Silverman} ch. 4 th. 6.1) that $\hat{E}(p\Zp)$ has no $p$-torsion if $p \geq 3$.

Since $E(K_{\fp_{\infty,i}})[p^{\infty}]=\{0\}$, therefore, as in point (1) above, the restriction map induces an isomorphism $H^1(K_{\fp_{n,i}}, E[p]) \isomarrow H^1(K_{\fp_{\infty, i}}, E[p])^{\Gamma_n}$. In addition from the local Euler-Poincar\'e characteristic (\cite{NSW} VII 7.3.1) and Tate local duality (\cite{NSW} VII 7.2.6) together with the Weil pairing we have $\dim_{\Fp}(H^1(K_{\fp_{n,i}}, E[p]))=2p^n+ 2 \dim_{\Fp}(E(K_{\fp_{n,i}})[p])$. But $E(K_{\fp_{n,i}})[p]=\{0\}$ and so $\dim_{\Fp}(H^1(K_{\fp_{n,i}}, E[p]))=2p^n$. Therefore we have shown that $\corank_{\overbar{\Lambda}}(H^1(K_{\fp_{\infty,i}}, E[p]))=2$.
But by point (2) above $E(K_{\fp_{\infty,i}}) \otimes \Fp$ is isomorphic to $H^1(K_{\fp_{\infty,i}}, E[p])$, so we also have $\corank_{\overbar{\Lambda}}(E(K_{\fp_{\infty,i}}) \otimes \Fp)=2.$

It follows that we have

$$\corank_{\overbar{\Lambda}}(E(K_{\fp_{\infty, 1}}) \otimes \Fp \times E(K_{\fp_{\infty, 2}}) \otimes \Fp)=4.$$\\
Therefore to show that $\corank_{\overbar{\Lambda}}(\coker \hat{\uptheta}) \le 2$ we only need to show that $\corank_{\overbar{\Lambda}}(\img \hat{\uptheta}) \geq 2$. This follows from conjecture C(ii) as is explained in the proof of theorem \ref{main_theorem2} in the next section: Consider the subgroup $M:=\dlim R_{2n} \alpha_{2n} + \dlim R_{2n+1} \alpha_{2n+1} \subseteq \Sel(E/K_{\infty})$.
The proof of theorem \ref{main_theorem2} shows that if conjecture C(ii) is true, then for $i=1$ or $i=2$ the image of $M$ under the map (induced by restriction)

$$E(K_{\infty}) \otimes \Fp \to E(K_{\fp_{\infty, i}}) \otimes \Fp$$\\
has $\overbar{\Lambda}$-corank greater than or equal to two. This implies the result.

\end{proof}

\begin{corollary}\label{control_theorem_corollary2}
Assume that $(E, \pi, p)$ satisfies $(*)$ then both $\Sel(E/K_{\infty})^{\; \dual}$ and $X_{s,p}(E/K_{\infty})$ are finitely generated $\overbar{\Lambda}$-modules
\begin{enumerate}[(a)]
\item If $E$ has ordinary reduction at $p$, then $\corank_{\overbar{\Lambda}}(\Sel(E/K_{\infty})) = \rank_{\overbar{\Lambda}}(X_{s,p}(E/K_{\infty})).$

\item If $E$ has supersingular reduction at $p$, $p$ splits in $K/\Q$ and conjecture C(ii) is true, then $\corank_{\overbar{\Lambda}}(\Sel(E/K_{\infty})) \le \rank_{\overbar{\Lambda}}(X_{s,p}(E/K_{\infty}))+2.$
\end{enumerate}
\end{corollary}
\begin{proof}
By \cite{Manin} th. 4.5 we know that $\Selinf(E/K_{\infty})^{\dual}$ is a finitely generated $\Lambda$-module. Since $E(K_{\infty})[p^{\infty}]=\{0\}$ by \cite{Matar} corollary 2.4, therefore we have an isomorphism $\Sel(E/K_{\infty}) \isomarrow \Selinf(E/K_{\infty})[p]$ and so $\Sel(E/K_{\infty})^{\dual}$ is a finitely generated $\overbar{\Lambda}$-module. Therefore by proposition \ref{Iwasawa_rank_proposition}, $Y^S_{f,p}(E/K_{\infty})$ is also a finitely generated $\overbar{\Lambda}$-module. Since by the control theorem we have an injection $X^S_{f,p}(E/K_{\infty}) \hookrightarrow Y^S_{f,p}(E/K_{\infty})$, therefore we also have that $X^S_{f,p}(E/K_{\infty})$ is a finitely generated $\overbar{\Lambda}$-module. The corollary now follows from the control theorem and proposition \ref{Iwasawa_rank_proposition}.
\end{proof}

\section{Proofs of Main Theorems}

Before proving the theorems listed in the introduction we record the following theorem which is essentially the main result proven in \cite{Matar}

\begin{theorem}\label{Iwasawa_rank_theorem}
Assume that $(E, \pi, p)$ satisfies $(*)$. Then we have
\begin{enumerate}
\item If $E$ has ordinary reduction at $p$, then $\rank_{\overbar{\Lambda}}(X_{s,p}(E/K_{\infty})) \le 1$.
\item If $E$ has supersingular reduction at $p$, $p$ splits in $K/\Q$ and conjecture $A$ is true, then $X_{s,p}(E/K_{\infty})=\{0\}$.
\end{enumerate}
\end{theorem}
\begin{proof}
In section 2.3 of \cite{Matar} (using the notation in that paper) we constructed a map $\uppsi_{\ell}: \dlim H^1(K_{n,\ell}, E)[p] \to X_p(E/K_{\infty})^{\dual}$ as follows:
First by Tate local duality we have an isomorphism

$$\dlim H^1(K_{n,\ell}, E)[p] \cong (\ilim E(K_{n,\ell})/p)^{\dual}.$$\\
Next for any $n$ we have a restriction map $\res_{\ell}: \Sel(E/K_n) \to E(K_{n, \ell})/p$. Taking
inverse limits gives a map $$\res_{\ell}: X_{s,p}(E/K_{\infty}) \to \ilim E(K_{n, \ell})/p$$\\
Dualizing this map and using the Tate duality isomorphism we get a map

$$\uppsi_{\ell}': \dlim H^1(K_{n, \ell}, E)[p] \to X_{s,p}(E/K_{\infty})^{\dual}.$$\\
Since $X_p(E/K_{\infty})$ injects into $X_{s,p}(E/K_{\infty})$, therefore we have a surjection $X_{s,p}(E/K_{\infty})^{\dual} \to X_p(E/K_{\infty})^{\dual}$ and so composing the map $\uppsi_{\ell}'$ with this surjection we get our desired map

$$\uppsi_{\ell}: \dlim H^1(K_{n,\ell}, E)[p] \to X_p(E/K_{\infty})^{\dual}.$$\\
If we work with the map $\uppsi_{\ell}'$ rather than $\uppsi_{\ell}$ we obtain results identical those in \cite{Matar} where the group $X_p(E/K_{\infty})$ gets replaced by $X_{s,p}(E/K_{\infty})$ so proposition 3.7 in the ordinary case gives that $\rank_{\overbar{\Lambda}}(X_{s,p}(E/K_{\infty})) \le 1$ and in the supersingular case the proof of theorem B in section 4 gives that $X_{s,p}(E/K_{\infty})=\{0\}$

However the reader should be aware of one important detail. In the beginning of sections 3 and 4 in \cite{Matar} we noted that the statements of theorems A and B were invariant under isogeny and therefore it would suffice to assume that $E$ is a strong Weil curve with a modular parametrization $\pi: J_0(N) \to E$ having a geometrically connected kernel. This was important to apply the results of Cornut \cite{Cornut} which require that $p$ does not divide the number of geometrically connected components of the kernel of the modular parametrization.

Regarding the theorem that we are proving the author has not been able to prove that it is invariant under isogeny and therefore we cannot pass to a strong Weil curve as we did in \cite{Matar} but as $(E, \pi, p)$ was assumed to satisfy $(*)$, therefore $p$ does not divide the number of geometrically connected components of the kernel $\pi: J_0(N) \to E$ and therefore the results of Cornut \cite{Cornut} apply to $E$ without the need to refer to a strong Weil curve. Also we have assumed in the supersingular case that conjecture A is satisfied (rather than conjecture A* in  \cite{Matar}) so we may work with the elliptic curve $E$ directly rather than working with an isogenous strong Weil curve.
\end{proof}

We now prove the theorems in the introduction

\begin{theorem}\label{main_theorem1}
Assume that $(E, \pi, p)$ satisfies $(*)$, $p$ splits in $K/\Q$ and $E$ has supersingular reduction at $p$, then conjecture A and conjecture B (for $\mathcal{E}=E$, $F=K$ and $l=p$) are equivalent.
\end{theorem}

\begin{proof}
Assume that $(E, \pi, p)$ satisfies $(*)$, $p$ splits in $K/\Q$, $E$ has supersingular reduction at $p$ and conjecture A is true.
Let $S$ be the set of primes of $K$ dividing $Np$. Since for any $n$ we have that $R^S_p(E/K_n)$ is contained in $\Sel(E/K_n)$, therefore $X^S_{f,p}(E/K_{\infty})$ is contained in $X_{s,p}(E/K_{\infty})$. But the latter group is trivial by theorem \ref{Iwasawa_rank_theorem} and so $X^S_{f,p}(E/K_{\infty})$ is trivial as well. Therefore it follows from corollary \ref{control_theorem_corollary1} that $R^S_p(E/K_{\infty})$ is finite. Now consider the natural map $R^S_p(E/K_{\infty}) \to R_{p^{\infty}}(E/K_{\infty})[p]$. We claim this map has a finite cokernel.

This follows from 3 facts. First, by \cite{Matar} corollary 2.4 we have $E(K_{\infty})[p^{\infty}]=\{0\}$ and therefore it follows that we have an isomorphism $H^1(K_{\infty}, E[p]) \isomarrow H^1(K_{\infty}, E[p^{\infty}])[p]$. Secondly, since we have assumed that all the primes dividing $N$ split in $K/\Q$ it follows from theorem 2 of \cite{Brink} that the set of primes of $K_{\infty}$ above $S$ is finite. Finally, it is easy to prove that for any prime $v$ of $K_{\infty}$ we have the kernel of the natural map $H^1(K_{\infty,v}, E[p]) \to H^1(K_{\infty,v}, E[p^{\infty}])[p]$ is finite.

The fact that the map $R^S_p(E/K_{\infty}) \to R_{p^{\infty}}(E/K_{\infty})[p]$ has a finite cokernel follows easily from these 3 facts. Therefore since $R^S_p(E/K_{\infty})$ is finite, we have that $R_{p^{\infty}}(E/K_{\infty})[p]$ is finite i.e. $R_{p^{\infty}}(E/K_{\infty})$ is cofinitely generated over $\Zp$ which proves conjecture B in this case.

Now assume that $(E, \pi, p)$ satisfies $(*)$, $p$ splits in $K/\Q$, $E$ has supersingular reduction at $p$ and conjecture B is true. Let us first make a few definitions. Let $\fp_1$ and $\fp_2$ be the 2 primes of $K$ above $p$. Since we have assumed that the class number of $K$ is prime to $p$ therefore both $\fp_1$ and $\fp_2$ are totally ramified in $K_{\infty}/K$. So in particular there are 2 primes $\fp_{n,1}$ and $\fp_{n,2}$ of $K_n$ above $p$ and 2 primes $\fp_{\infty, 1}$ and $\fp_{\infty, 2}$ of $K_{\infty}$ above $p$. We will denote the completion of $K_n$ with respect to $\fp_{n,i}$ by $K_{\fp_{n,i}}$ and we let $K_{\fp_{\infty,i}}$ be the union of the completions $K_{\fp_{n,i}}$. Following Kobayashi \cite{Kobayashi},  we define the following subgroups of $E(K_{\fp_{n,i}})$

$$E^+(K_{\fp_{n,i}}):=\{x \in E(K_{\fp_{n,i}}) \; | \; \text{Tr}_{n/m+1}(x) \in E(K_{\fp_{m,i}}) \; \text{for even}\; m : \; 0 \le m < n \}$$
$$E^-(K_{\fp_{n,i}}):=\{x \in E(K_{\fp_{n,i}}) \; | \; \text{Tr}_{n/m+1}(x) \in E(K_{\fp_{m,i}}) \; \text{for odd}\; m : \; 0 \le m < n \}.$$\\

We then define $E^+(K_{\fp_{\infty,i}}):= \bigcup_{n \in \N}$ $E^+(K_{\fp_{n,i}})$ and $E^-(K_{\fp_{\infty,i}}):= \bigcup_{n \in \N} E^-(K_{\fp_{n,i}}).$

We now analyze the intersection of $E^+(K_{\fp_{\infty,i}}) \otimes \Fp$ and $E^-(K_{\fp_{\infty,i}}) \otimes \Fp$ where we view both of these groups as subgroups of $E(K_{\fp_{\infty,i}}) \otimes \Fp$. By a proof which is identical to that of lemma 2.6.5 of \cite{CW}, using a result of Iovita and Pollack \cite{IP}, we have

\begin{equation}\label{intersection_equality}
E^+(K_{\fp_{\infty,i}}) \otimes \Fp \cap E^-(K_{\fp_{\infty,i}}) \otimes \Fp = E(\Qp) \otimes \Fp
\end{equation}\\
We now define some Selmer groups. First let $S$ be the set of primes of $K$ above $p$ and above all the primes dividing $N$. We let $K_S$ be the maximal extension of $K$ unramified outside $S$. For any field $F$ with $K \subseteq F \subseteq K_S$ we let $G_S(F)=\Gal(K_S/F)$ and we let $S_F$ be the set of primes of $F$ that lie over a prime of $S$.

Let $\tilde{S}_{K_{\infty}}=S_{K_{\infty}} \backslash \{\fp_{\infty,1}, \fp_{\infty,2}\}$. Since we have assumed that all the primes dividing $N$ split in $K/\Q$, therefore it follows from theorem 2 of \cite{Brink} that the set $\tilde{S}_{K_{\infty}}$ is finite.

Recall that the $p$-Selmer group of $E$ over $K_{\infty}$ is defined as

$$\displaystyle 0 \longrightarrow \Sel(E/K_{\infty}) \longrightarrow H^1(G_S(K_{\infty}), E[p])\longrightarrow \prod_{v\in S_{K_{\infty}}} \frac{H^1(K_{\infty, v}, E[p])}{E(K_{\infty,v})\otimes \Fp}.$$

Following Kobayashi \cite{Kobayashi}, we define the even (odd) $p$-Selmer group of $E$ over $K_{\infty}$ as

$$\displaystyle 0 \longrightarrow \Sel^{\pm}(E/K_{\infty}) \longrightarrow \Sel(E/K_{\infty}) \longrightarrow \prod_{i=1,2} \frac{H^1(K_{\fp_{\infty}, i}, E[p])}{E^{\pm}(K_{\fp_{\infty}, i})\otimes \Fp}.$$

We also define

$$\displaystyle 0 \longrightarrow \Selm^1_p(E/K_{\infty}) \longrightarrow \Sel(E/K_{\infty}) \longrightarrow \prod_{i=1,2} \frac{H^1(K_{\fp_{\infty}, i}, E[p])}{E(\Qp)\otimes \Fp}.$$\\
We are now ready to show that conjecture A is true in this case. According to theorem 4.1 of \cite{Matar} neither $\dlim R_{2n} \alpha_{2n}$ nor $\dlim R_{2n+1} \alpha_{2n+1}$ is $\overbar{\Lambda}$-cotorsion. Therefore the conjecture will be proven if we show that $\dlim R_{2n} \alpha_{2n} \cap \dlim R_{2n+1} \alpha_{2n+1}$ is finite.

Since $\text{Tr}_{K_{n+1}/K_n}(\alpha_{n+1})=-\alpha_{n-1}$ therefore we have that $\res_{\fp_{2n,i}} \alpha_{2n} \in E^+(K_{\fp_{2n,i}})$ and $\res_{\fp_{2n+1, i}} \alpha_{2n+1} \in E^-(K_{\fp_{2n+1, i}})$. This implies that $\dlim R_{2n} \alpha_{2n} \subseteq \Sel^+(E/K_{\infty})$ and $\dlim R_{2n+1} \alpha_{2n+1} \subseteq \Sel^-(E/K_{\infty})$ and so it suffices to show that $\Sel^+(E/K_{\infty}) \cap \Sel^-(E/K_{\infty})$ is finite. But by (\ref{intersection_equality}) above this intersection is $\Selm^1_p(E/K_{\infty})$.

Now define $$\mathcal{L}(K_{\infty})= \prod_{\makebox[0pt]{$\scriptstyle i=1,2$}} E(\Qp)\otimes \Fp \times \prod_{v \in \tilde{S}_{K_{\infty}}} E(K_{\infty,v})\otimes \Fp.$$

We claim that this group is finite. First of all, by Mattuck's theorem we have that $E(\Qp) \cong \Zp \times T$ where $T$ is finite group. Therefore $E(\Qp) \otimes \Fp$ is finite. Now let $v \in \tilde{S}_{K_{\infty}}$ and let $l \neq p$ be the rational prime below $v$. For any $n$ we will also let $v$ denote the prime of $K_n$ below $v$. By Mattuck's theorem we have $E(K_{n,v}) \cong \mathbb{Z}_l^r \times T$ where $r$ is some integer and $T$ is a finite group. Therefore $\# (E(K_{n,v}) \otimes \Fp) \le p^2$. It follows that $E(K_{\infty,v}) \otimes \Fp$ is finite. Since $\tilde{S}_{K_{\infty}}$ is finite, we have shown that $\mathcal{L}(K_{\infty})$ is in fact finite.

Now consider the following commutative diagram

\begin{equation}
\xymatrix{
& 0 \ar[d] \ar[r] & H^1(G_S(K_{\infty}), E[p]) \ar[d] \ar[r] & H^1(G_S(K_{\infty}), E[p]) \ar[d] \ar[r]  & 0 \\
0 \ar[r] & \mathcal{L}(K_{\infty}) \ar[r] & \displaystyle \prod_{\mathclap{v \in S_{K_{\infty}}}} H^1(K_{\infty,v}, E[p]) \ar[r] & \displaystyle \prod_{\mathclap{v \in S_{K_{\infty}}}} H^1(K_{\infty, v}, E[p])/ \mathcal{L}(K_{\infty}) \ar[r] & 0
}
\end{equation}

Applying the snake lemma to this diagram we get an exact sequence

$$0 \longrightarrow R^S_p(E/K_{\infty}) \longrightarrow \Selm^1_p(E/K_{\infty}) \longrightarrow \mathcal{L}(K_{\infty})$$\\
Since $\mathcal{L}(K_{\infty})$ is finite, the exact sequence shows that $R^S_p(E/K_{\infty})$ is finite if and only if $\Selm^1_p(E/K_{\infty})$ is finite. Therefore it suffices to show that $R^S_p(E/K_{\infty})$ is finite. To show this, we note that by \cite{Matar} corollary 2.4 we have $E(K_{\infty})[p^{\infty}]=\{0\}$ from which it follows that the natural map $R^S_p(E/K_{\infty}) \to R_{p^{\infty}}(E/K_{\infty})[p]$ is an injection. Since $R_{p^{\infty}}(E/K_{\infty})$ is cofinitely generated over $\Zp$, therefore $R_{p^{\infty}}(E/K_{\infty})[p]$ is finite. This in turn implies that $R^S_p(E/K_{\infty})$ is finite which completes the proof.
\end{proof}

\begin{theorem}\label{main_theorem2}
Assume that $(E, \pi, p)$ satisfies $(*)$ then we have
\begin{enumerate}[(a)]
\item If $E$ has ordinary reduction at $p$, then conjecture C(i) and conjecture B (for $\mathcal{E}=E$, $F=K$ and $l=p$) are equivalent.
\item If $p$ splits in $K/\Q$ and $E$ has supersingular reduction at $p$, then conjecture C(ii) and conjecture B (for $\mathcal{E}=E$, $F=K$ and $l=p$) are equivalent (and hence also equivalent to conjecture A by the previous theorem).
\end{enumerate}
\end{theorem}
\begin{proof}
First we prove (a): Assume that $E$ has ordinary reduction at $p$ and conjecture C(i) is true. Consider the module $\dlim R_n \alpha_n \subseteq \Sel(E/K_{\infty})$. By \cite{Matar} theorem 3.1, theorem \ref{Iwasawa_rank_theorem} and corollary \ref{control_theorem_corollary2} we have that both $\dlim R_n \alpha_n$ and $\Sel(E/K_{\infty})$ are finitely generated $\overbar{\Lambda}$-modules and that $\corank_{\overbar{\Lambda}}(\dlim R_n \alpha_n) \geq 1$ and $\corank_{\overbar{\Lambda}}(\Sel(E/K_{\infty})) \le 1$. Since $\dlim R_n \alpha_n$ is contained in $\Sel(E/K_{\infty})$, therefore it follows that both their $\overbar{\Lambda}$-coranks must be equal to one.

Now let $S$ be the set of primes of $K$ dividing $Np$. By the same argument in the proof of theorem \ref{main_theorem1}, we have that $R^S_p(E/K_{\infty}) \to R_{p^{\infty}}(E/K_{\infty})[p]$ has a finite cokernel and therefore to prove conjecture B we only have to show that $R^S_p(E/K_{\infty})$ is finite. Since $R^S_p(E/K_{\infty}) \subseteq \Sel(E/K_{\infty})$ and both $\dlim R_n \alpha_n$ and $\Sel(E/K_{\infty})$ have $\overbar{\Lambda}$-coranks equal to one, therefore it is is easy to see that the finiteness of $R^S_p(E/K_{\infty})$ will follow if we can show that $M:=\dlim R_n \alpha_n \cap R^S_p(E/K_{\infty})$ is finite.

Assume on the contrary that $M$ is infinite. Since finitely generated torsion $\overbar{\Lambda}$-modules are finite and $\corank_{\overbar{\Lambda}}(\dlim R_n \alpha_n)=1$, therefore $M$ must also have $\overbar{\Lambda}$-corank equal to one. Therefore $\dlim R_n \alpha_n /M$ is finite and so is annihilated by $g^{p^m}-1$ for some $m \in \N$ where $g$ is a topological generator of $\Gamma$ i.e. $(g^{p^m}-1)\dlim R_n \alpha_n \subset R^S_p(E/K_{\infty})$.

Let $\fp$ be a prime of $K$ above $p$. Since we have assumed that the class number of $K$ is prime to $p$, therefore $\fp$ is totally ramified in $K_{\infty}/K$. For any $n$ let $\fp_n$ be the prime of $K_n$ above $\fp$. Since $(g^{p^m}-1)\dlim R_n \alpha_n \subseteq R^S_p(E/K_{\infty})$, therefore by the definition of the fine $p$-Selmer group we have $(g^{p^m}-1)\dlim R_n \res_{\fp_n} \alpha_n = 0$ and so $\dlim R_n \res_{\fp_n} \alpha_n$ is $\overbar{\Lambda}$-cotorsion. This contradicts conjecture C(i) which completes the proof of the forward implication of part (a) of the theorem.

Now assume that $E$ has ordinary reduction at $p$ and conjecture B is true. We must show that conjecture C(i) is true. Arguing by contradiction, assume that conjecture C(i) is false i.e. that for every prime $\fp_{\infty}$ of $K_{\infty}$ above $p$ the $\Gamma$-submodule of $E(K_{\fp_{\infty}})/p$ generated by the Heegner points $\alpha_n$ is finite. Let $S$ be the set of primes of $K$ dividing $Np$ and let $g$ be a topological generator of $\Gamma$. We claim that there exists a $k \in \N$ such that $(g^{p^k}-1)\dlim R_n \alpha_n \subseteq R^S_p(E/K_{\infty})$. From the definition of $R^S_p(E/K_{\infty})$, to show this, we need to prove that there exists a $k \in \N$ such that for all $s \in \dlim R_n \alpha_n$ we have $\res_v((g^{p^k}-1)s)=0$ for any $v \in S_{K_{\infty}}$. Note that since we have assumed that all the primes dividing $N$ to split in $K/\Q$ therefore it follows from \cite{Brink} th. 2 that $S_{K_{\infty}}$ is finite.

Let $v \in S_{K_{\infty}}$ be a prime not dividing $p$. Since the set $S_{K_{\infty}}$ is finite, therefore the decomposition group of $v$ in $\Gamma$ is nontrivial and so is of the form $\Gamma^{p^m}$ for some $m$. Now let $s \in \dlim R_n \alpha_n$. Then $\res_v(s) \in E(K_{\infty,v})/p$ and we claim that $E(K_{\infty,v})/p$ is finite. To see this, let $l \neq p$ be the rational prime below $v$. For any $n$ we will also let $v$ denote the prime of $K_n$ below $v$. By Mattuck's theorem we have $E(K_{n,v}) \cong \mathbb{Z}_l^r \times T$ where $r$ is some integer and $T$ is a finite group. Therefore $\# (E(K_{n,v})/p) \le p^2$. It follows from this that  $E(K_{\infty,v})/p$ is in fact finite as claimed. The decomposition group $\Gamma^{p^m}$ acts on the finite group $E(K_{\infty,v})/p$ so there exists $k_v \geq m$ such that $(g^{p^{k_v}}-1)E(K_{\infty,v})/p=0$. It follows that we have $\res_v((g^{p^{k_v}}-1)s)=(g^{p^{k_v}}-1)\res_v(s)=0$.

Now let $v \in S_{K_{\infty}}$ be a prime above $p$. Since we have assumed that the class number of $K$ is relatively prime to $p$, therefore every prime of $K$ above $p$ is totally ramified in $K_{\infty}/K$. Let $s \in \dlim R_n \alpha_n$. Since conjecture C(i) is false therefore the $\Gamma$-submodule of $E(K_{\infty,v})/p$ generated by the points $\alpha_n$ is finite and so there exists $k_v \in N$ such that $g^{p^{k_v}}-1$ that annihilates this submodule. It follows that $\res_v((g^{p^{k_v}}-1)s)=(g^{p^{k_v}}-1)\res_v(s)=0$.

We have shown that for every $v \in S_{K_{\infty}}$ there exists $k_v \in \N$ such that $\res_v((g^{p^{k_v}}-1)s)=0$ for any $s \in \dlim R_n \alpha_n$. Then taking $k$ to be the maximum of the integers $k_v$ we get $\res_v((g^{p^k}-1)s)=0$ for all $s \in \dlim R_n \alpha_n$ and any $v \in S_{K_{\infty}}$. This implies that $(g^{p^m}-1)\dlim R_n \alpha_n \subseteq R^S_p(E/K_{\infty})$ as desired. By theorem 3.1 of \cite{Matar}, $\dlim R_n \alpha_n$ has $\overbar{\Lambda}$-corank greater than or equal to one. Therefore $(g^{p^m}-1)\dlim R_n \alpha_n$ also has $\overbar{\Lambda}$-corank greater than or equal to one and as this group is contained in $R^S_p(E/K_{\infty})$ it follows that $R^S_p(E/K_{\infty})$ is infinite. Now by corollary 2.4 of \cite{Matar}, $E(K_{\infty})[p^{\infty}]=\{0\}$ so the natural map $R^S_p(E/K_{\infty}) \to R_{p^{\infty}}(E/K_{\infty})[p]$ is an injection. This proves that $R_{p^{\infty}}(E/K_{\infty})[p]$ is infinite i.e. that $R_{p^{\infty}}(E/K_{\infty})$ is not cofinitely generated over $\Zp$. This contradicts our assumption that conjecture B is true which thereby proves the backward implication of part (a).

We now prove part (b): First we prove the forward implication. Assume that $p$ splits in $K/\Q$, $E$ has supersingular reduction at $p$ and conjecture C(ii) is true. Then there exists a prime $\fp_{\infty}$ of $K_{\infty}$ above $p$ such that both $\dlim R_{2n} \res_{\fp_{2n}} \alpha_{2n}$ and $\dlim R_{2n+1} \res_{\fp_{2n+1}} \alpha_{2n+1}$ are infinite (where $\fp_n$ be the prime of $K_n$ below $\fp_{\infty}$). In what follows let $K_{\fp_n}$ be the completion of $K_n$ with respect to $\fp_n$.

To prove conjecture B is true, it suffices by theorem \ref{main_theorem1} to prove that conjecture A is true. To prove this, it clearly suffices to show that $\dlim R_{2n} \res_{\fp_{2n}}\alpha_{2n} + \dlim R_{2n+1} \res_{\fp_{2n+1}}\alpha_{2n+1}$ has $\overbar{\Lambda}$-corank greater than or equal to two (note that both $\dlim R_{2n} \res_{\fp_{2n}}\alpha_{2n}$ and $\dlim R_{2n+1} \res_{\fp_{2n+1}}\alpha_{2n+1}$ are cofinitely generated $\overbar{\Lambda}$-modules since both $\dlim R_{2n} \alpha_{2n}$ and $\dlim R_{2n+1} \alpha_{2n+1}$ are cofinitely generated $\overbar{\Lambda}$-modules by the argument in theorem 4.1 of \cite{Matar}).

Since finitely generated torsion $\overbar{\Lambda}$-modules are finite, therefore conjecture C(ii) implies that both $\dlim R_{2n} \res_{\fp_{2n}} \alpha_{2n}$ and $\dlim R_{2n+1} \res_{\fp_{2n+1}} \alpha_{2n+1}$ have $\overbar{\Lambda}$-coranks greater than or equal to one so to prove conjecture A we only have to show that $X:=\dlim R_{2n} \res_{\fp_{2n}} \alpha_{2n} \cap \dlim R_{2n+1} \res_{\fp_{2n+1}} \alpha_{2n+1}$ is finite. To show this, we use the same argument as in theorem \ref{main_theorem1}.

Following Kobayashi \cite{Kobayashi} we define the following subgroups of $E(K_{\fp_n})$

$$E^+(K_{\fp_n}):=\{x \in E(K_{\fp_n}) \; | \; \text{Tr}_{n/m+1}(x) \in E(K_{\fp_m}) \; \text{for even}\; m : \; 0 \le m < n \}$$
$$E^-(K_{\fp_n}):=\{x \in E(K_{\fp_n}) \; | \; \text{Tr}_{n/m+1}(x) \in E(K_{\fp_m}) \; \text{for odd}\; m : \; 0 \le m < n \}.$$\\
We then define $E^+(K_{\fp_{\infty}}):= \bigcup_{n \in \N}$ $E^+(K_{\fp_n})$ and $E^-(K_{\fp_{\infty}}):= \bigcup_{n \in \N} E^-(K_{\fp_n}).$

We now analyze the intersection of $E^+(K_{\fp_{\infty}}) \otimes \Fp$ and $E^-(K_{\fp_{\infty}}) \otimes \Fp$ where we view both of these groups as subgroups of $E(K_{\fp_{\infty}}) \otimes \Fp$. By a proof identical to that of lemma 2.6.5 of \cite{CW}, using a result of Iovita and Pollack \cite{IP}, we have

$$E^+(K_{\fp_{\infty}}) \otimes \Fp \cap E^-(K_{\fp_{\infty}}) \otimes \Fp = E(\Qp) \otimes \Fp.$$\\
Since $\text{Tr}_{K_{n+1}/K_n}(\alpha_{n+1})=-\alpha_{n-1}$, therefore it follows that $\dlim R_{2n} \res_{\fp_{2n}} \alpha_{2n} \subseteq E^+(K_{\fp_{\infty}})$ and $\dlim R_{2n+1} \res_{\fp_{2n+1}} \alpha_{2n+1} \subseteq E^-(K_{\fp_{\infty}})$ so we have

$$X \subseteq E^+(K_{\fp_{\infty}}) \cap E^-(K_{\fp_{\infty}})=E(\Qp) \otimes \Fp.$$\\
But by Mattuck's theorem, $E(\Qp)\cong \Zp \times T$ where $T$ is a finite group. Therefore $E(\Qp) \otimes \Fp$ is finite which in turn makes $X$ finite. This completes the proof of the forward implication of part (b).

Now assume that $p$ splits in $K/\Q$, $E$ has supersingular reduction at $p$ and conjecture B is true. We will prove that backward implication i.e. that conjecture C(ii) is true. The proof goes along the same lines as the backward implication of part (a), however since we have to deal with the Heegner points $\alpha_{2n}$ and the points $\alpha_{2n+1}$ separately the proof is not as straightforward.

Arguing by contradiction, assume that conjecture C(ii) is false. Since $p$ splits in $K/\Q$ and the class number of $K$ is prime to $p$ therefore there are 2 primes $\fp_{\infty}$ and $\bar{\fp}_{\infty}$ of $K_{\infty}$ above $p$. Since conjecture C(ii) is false, either the $\Gamma$-submodule of $E(K_{\fp_{\infty}})/p$ generated by the points $\alpha_{2n}$ or the $\Gamma$-submodule generated by the points $\alpha_{2n+1}$ is finite. Let us assume that the the former submodule is finite (if the former submodule is infinite and the latter is finite our proof will be very similar). Then there exists an $k \in \N$ such that $(g^{p^k}-1)\dlim R_{2n} \res_{\fp_{2n}} \alpha_{2n}=0$ ($\fp_n$ is the prime of $K_n$ below $\fp_{\infty}$).

Now let $S$ be the set of primes of $K$ dividing $Np$. We claim that for some $m \in \N$ we have $(g^{p^m}-1)\dlim R_{2n} \alpha_{2n} \subseteq R^S_p(E/K_{\infty})$. By the proof of the backward implication of part (a), we see that to prove this it suffices to show that $\res_v((g^{p^k}-1)s)=0$ for any $s \in \dlim R_{2n} \alpha_{2n}$ any $v \in \{\fp_{\infty}, \bar{\fp}_{\infty}\}$. Let $\tau$ be a complex conjugation of $K_{\infty}$ so that $\tau\fp_{\infty}=\bar{\fp}_{\infty}$. Since $(g^{p^k}-1)\dlim R_{2n} \res_{\fp_{2n}} \alpha_{2n}=0$, therefore for any $s \in \dlim R_{2n} \alpha_{2n}$ we have $\res_{\fp_{\infty}}((g^{p^k}-1)s)=0$ so we only have to show that $\res_{\bar{\fp}_{\infty}}((g^{p^k}-1)s)=0$.

The automorphism $\tau$ induces a ``change of group'' automorphism $\tau_*$ on $H^1(K_{\infty}, E[p])$ and an isomorphism $\tau_*: H^1(K_{\fp_{\infty}}, E[p]) \to H^1(K_{\bar{\fp}_{\infty}}, E[p])$. For any $s \in H^1(K_{\infty}, E[p])$ we have $\res_{\bar{\fp}_{\infty}}(\tau_*(s))=\tau_*(\res_{\fp_{\infty}}(s))$. From this we see that to show that $\res_{\bar{\fp}_{\infty}}((g^{p^k}-1)s)=0$ for all $s \in \dlim R_{2n} \alpha_{2n}$, we only have to show that $(g^{p^k}-1)\dlim R_{2n} \alpha_{2n}$ is $\tau$-invariant.

Let $M=\dlim R_{2n} \alpha_{2n}$ and denote for any $t \in \N$ the group $\Gamma^{p^t}$ by $\Gamma_t$. Clearly to show that $(g^{p^k}-1)M$ is $\tau$-invariant it suffices to show that $(g-1)^{p^k}M^{\Gamma_t}=(g^{p^k}-1)M^{\Gamma_t}$ is $\tau$-invariant for any $t$ (note that $(g-1)^{p^k} \equiv g^{p^k}-1 \mod p)$. Recall that $\tau g \tau = g^{-1}$. Therefore we have $\tau (g-1)^{p^k} M^{\Gamma_t} = (g^{-1} -1)^{p^k} \tau M^{\Gamma_t}$. Again since $\tau g \tau =g^{-1}$, therefore it follows that $\tau M^{\Gamma_t} \subseteq M^{\Gamma_t}$ so our desired result will follow if we can can show that $(g^{-1}-1)^{p^k} M^{\Gamma_t} \subseteq (g-1)^{p^k} M^{\Gamma_t}$. But $\Gamma/\Gamma_t$ has order $p^t$ so therefore $(g^{-1}-1)^{p^k} M^{\Gamma_t}=(g^{p^t-1}-1)^{p^k} M^{\Gamma_t}$ and the desired result follows since $g-1$ divides $g^{p^k-1}-1$.

We have shown that $(g^{p^m}-1)\dlim R_{2n} \alpha_{2n} \subseteq R^S_p(E/K_{\infty})$ for some $m \in \N$. By theorem 4.1 of \cite{Matar} $\dlim R_{2n} \alpha_{2n}$ has $\overbar{\Lambda}$-corank greater than or equal to one. Therefore $(g^{p^m}-1)\dlim R_{2n} \alpha_{2n}$ also has $\overbar{\Lambda}$-corank greater than or equal to one and as this group is contained in $R^S_p(E/K_{\infty})$, it follows that $R^S_p(E/K_{\infty})$ is infinite. Now by corollary 2.4 of \cite{Matar} $E(K_{\infty})[p^{\infty}]=\{0\}$ so the natural map $R^S_p(E/K_{\infty}) \to R_{p^{\infty}}(E/K_{\infty})[p]$ is an injection. This proves that $R_{p^{\infty}}(E/K_{\infty})[p]$ is infinite i.e. that $R_{p^{\infty}}(E/K_{\infty})$ is not cofinitely generated over $\Zp$. This contradicts our assumption that conjecture B is true which thereby proves the backward implication of part (b).

\end{proof}

\begin{theorem}
Suppose that $(E, \pi, p)$ satisfies $(*)$ then we have
\begin{enumerate}[(a)]
\item If $E$ has ordinary reduction at $p$, then $\Selinf(E/K_{\infty})^{\dual}$ has $\Lambda$-rank equal to 1 and $\mu$-invariant equal to zero.
\item If $p$ splits in $K/\Q$, $E$ has supersingular reduction at $p$ and conjecture C(ii) is true, then $\Selinf(E/K_{\infty})^{\dual}$ has $\Lambda$-rank equal to 2 and $\mu$-invariant equal to zero.
\end{enumerate}
\end{theorem}
\begin{proof}
Note that by theorems \ref{main_theorem1} and \ref{main_theorem2}, conjecture C(ii) implies conjecture A (also see the comments made in the proof of theorem \ref{Iwasawa_rank_theorem} about conjecture A versus conjecture A*). Therefore by theorems A and B in \cite{Matar},  $\Sel(E/K_{\infty})^{\dual}$ has $\Lambda$-rank 1 in the ordinary case (part (a)) and rank 2 in supersingular case (part (b)). By corollary 2.4 in \cite{Matar} we have that $E(K_{\infty})[p^{\infty}]=\{0\}$ and therefore we have an isomorphism $\Sel(E/K_{\infty}) \isomarrow \Selinf(E/K_{\infty})[p]$. From this and the value of the $\Lambda$-corank of $\Selinf(E/K_{\infty})$, we see that to show that $\Selinf(E/K_{\infty})^{\dual}$ has $\mu$-invariant equal to zero we need to show that the $\overbar{\Lambda}$-corank of $\Sel(E/K_{\infty})$ is less than or equal to one in the ordinary case (part (a)) and less than or equal to two in the supersingular case (part (b)). This follows from theorem \ref{Iwasawa_rank_theorem} and corollary \ref{control_theorem_corollary2}.
\end{proof}

\section{Examples verifying Conjecture B}
In the setup for conjecture B let $\Gamma=\Gal(F^{anti}/F)$ and $\Lambda=\Zl[[\Gamma]]$ the corresponding Iwasawa algebra. Conjecture B predicts that  $R_{l^{\infty}}(\mathcal{E}/F^{anti})$ is cofinitely generated over $\Zl$. It is easy to see that this is equivalent to both of the following statements

\noindent (a) $R_{l^{\infty}}(\mathcal{E}/F^{anti})^{\dual}$ is a torsion $\Lambda$-module\\
(b) $R_{l^{\infty}}(\mathcal{E}/F^{anti})^{\dual}$ has $\mu$-invariant equal to zero\\

Statement (a) is equivalent to $H^2(G_S(F^{anti}), E[p^{\infty}])=0$ (see \cite{CS} lemma 3.1) and is usually called the weak Leopoldt conjecture. It has been proven by Bertolini \cite{Bertolini2} when $\mathcal{E}$ is defined over $\Q$ and $l$ is prime where $E$ has good ordinary reduction (together with some additional conditions listed in that paper). When $E$ has supersingular reduction at $l$ and $l$ splits in $F/\Q$ then the weak Leopoldt conjecture is also true. This follows from \cite{Ciperiani} theorem 3.1 and \cite{IP} theorem 6.1.

In this section we will produce examples verifying conjecture B. The examples will be constructed using the following theorem whose proof relies fundamentally on the work of Wuthrich \cite{Wuthrich1}.

Before stating the theorem, consider an elliptic curve $E$ defined over a number field $K$ and let $v$ be a prime of $K$. If $K_v$ denotes the completion of $K$ at $v$ then, as is standard, we let $E_0(K)$ denote the subgroup of $E(K_v)$ with nonsingular reduction modulo $v$ and $E_1(K_v)$ the subgroup of points in $E_0(K_v)$ that reduce to the identity. Write $c_v=[E_0(K):E_1(K)]$ and write $\log_{E,v}: E_1(K_v) \to K_v$ for the formal logarithm map. This map depends on our choice of a minimal Weierstrass equation for $E$ over $K_v$, but its values are well-defined up to multiplication by a unit in $O_v$, the ring of integers of $K_v$. Finally, let $i_v: E(K) \to E(K_v)$ be the natural embedding.

\begin{theorem}\label{Euler_char_theorem}
Let $E$ be an elliptic curve of conductor $N$ defined over $\Q$ and let $K$ be an imaginary quadratic field of discriminant $d_K$ such that all the primes dividing $N$ split in $K/\Q$. Let $p \nmid Nd_K$ be an odd rational prime such that $\Gal(\Q(E[p])/\Q)=GL_2(\Fp)$. For this prime $p$, let $K_{\infty}/K$ be the anticyclotomic $\Zp$-extension of $K$. Now if $K[1]$ is the Hilbert class field of $K$, then a choice of an ideal $\cN$ of $\cO_K$ such that $\cO_K/\cN \cong \cy{N}$ and a modular parametrization $X_0(N) \to E$ allows us to define a Heegner point $y_1 \in E(K[1])$. Let $y_K$ be the trace of this point down to $K$. We make the following assumptions

\begin{enumerate}[(a)]
\item The Heegner point $y_K \in E(K)$ has infinite order (equivalent to $L'(E/K,1) \neq 0$ by Zhang's \cite{Zhang} generalization to the  Gross-Zagier theorem \cite{GZ} \footnote{Gross and Zagier assume in their paper that the discriminant of $K$ is odd. This assumption is removed by Zhang}).
\item $p$ does not divide the image of the $y_K$ in $E(K)/E(K)_{tors}$
\item For every prime $v$ of $K$ dividing $p$ we have $p \nmid \#\tilde{E}(k_v)$ (where $k_v$ is the residue field of $K_v$ and $\tilde{E}$ is the reduced curve)
\item For every prime $v$ of $K$ we have $p \nmid c_v$
\item There exists a prime $v$ of $K$ above $p$ and a point $P \in \cap_{w |p} i_w^{-1}(E_1(K_w))$ such that $\ord_v(\log_{E,v}(i_v(P)))=1$
\end{enumerate}
Under the above assumptions, we have that $R_{p^{\infty}}(E/K_{\infty})$ is cofinitely generated over $\Zp$.
\end{theorem}
\begin{proof}
The proof of this theorem relies on the work of Wuthrich \cite{Wuthrich1} on the Euler characteristic of the fine Selmer group. First of all, let us consider the compact version of the fine Selmer group. Let $S$ be the finite set of primes of $K$ dividing $p$ and where $E$ has bad reduction. Let $K_S$ be the maximal extension of $K$ unramified outside of $S$ and $G_S(K)=\Gal(K_S/K)$. The compact fine Selmer group $\mathfrak{R}_p(E/K)$ is defined by the following exact sequence

$$\displaystyle 0 \longrightarrow \mathfrak{R}_p(E/K) \longrightarrow H^1(G_S(K), T_p E) \longrightarrow \prod_{v | p} H^1(K_v, T_p E).$$

In the above and what follows, if $M$ is an abelian group, then $T_p M$ will denote its $p$-adic Tate module.

We claim that $\mathfrak{R}_p(E/K)$ is trivial. To see this, first note that by the proof of lemma 3.1 in Wuthrich's paper we have that $\mathfrak{R}_p(E/K)$ injects into $T_p R_{p^{\infty}}(E/K)$ so we only have to show that $T_p R_{p^{\infty}}(E/K)$ is trivial. To prove this, note that by \cite{Wuthrich2} we have an exact sequence
\begin{equation}\label{exact_sequence}
\displaystyle 0 \longrightarrow M_{p^{\infty}}(E/K) \longrightarrow R_{p^{\infty}}(E/K) \longrightarrow \FSha_{p^{\infty}}(E/K) \longrightarrow 0
\end{equation}
where $M_{p^{\infty}}(E/K)$ is defined as $$0 \longrightarrow M_{p^{\infty}}(E/K) \longrightarrow E(K)\otimes \Qp/\Zp \longrightarrow \prod_{v |p} E(K_v) \otimes \Qp/\Zp.$$
and $\FSha_{p^{\infty}}(E/K)$ is the fine Tate-Shafarevich group defined as simply $\FSha_{p^{\infty}}(E/K)=R_{p^{\infty}}(E/K)/M_{p^{\infty}}(E/K)$. It is shown in \cite{Wuthrich2} that $\FSha_{p^{\infty}}(E/K)$ is a subgroup of $\Sha(E/K)[p^{\infty}]$.

From equation (\ref{exact_sequence}) above, to show that $T_p R_{p^{\infty}}(E/K)$ is trivial it suffices to show that both $T_p M_{p^{\infty}}(E/K)$ and $T_p \FSha_{p^{\infty}}(E/K)$ are trivial which we now show. Condition (a) of the theorem implies by the work of Kolyvagin \cite{Kolyvagin} that $E(K)$ has rank 1 and that $\Sha(E/K)$ is finite. Since $\Sha(E/K)$ is finite and $\FSha_{p^{\infty}}(E/K)$ is a subgroup of $\Sha(E/K)[p^{\infty}]$ therefore $T_p \FSha_{p^{\infty}}(E/K)$ is trivial. As for $T_p M_{p^{\infty}}(E/K)$, the fact that it is trivial follows easily from the definition of $M_{p^{\infty}}(E/K)$, and the fact that the rank of $E(K)$ is 1 and that by Mattuck's theorem for any prime $v$ of $K$ above $p$ we have $E(K_v) \cong \Zp^{[K_v:\Q_p]} \times T$ where $T$ is a finite group. Thus we have shown that $T_p R_{p^{\infty}}(E/K)$ is trivial which as we explained above proves that $\mathfrak{R}_p(E/K)$ is trivial.

We now return to Wuthrich's paper \cite{Wuthrich1} and refer to it for the rest of the proof. Attached to the anticyclotomic $\Zp$-extension $K_{\infty}/K$ Wuthrich defines a $p$-adic height pairing (he defines a pairing for any $\Zp$-extension of $K$):

$$\langle \; , \; \rangle_{K_{\infty}}: \mathfrak{R}_p(E/K) \times \mathfrak{R}_p(E/K) \longrightarrow \Qp.$$\\
Since $\mathfrak{R}_p(E/K)$ is trivial, therefore the above pairing is non-degenerate. This implies by theorem 6.1 in Wuthrich's paper that $R_{p^{\infty}}(E/K_{\infty})^{\dual}$ is a torsion $\Lambda$-module where $\Lambda$ is the Iwasawa algebra attached to the extension $K_{\infty}/K$. Let $\Gamma=\Gal(K_{\infty}/K)$. Choosing a topological generator $\gamma \in \Gamma$ allows us to identify the Iwasawa algebra $\Lambda$ with $\Zp[[T]]$ (via an isomorphism mapping $\gamma-1$ to $T$) and so with this choice of a topological generator $\gamma$  we may define the characteristic polynomial $f_R(T)$ of the torsion $\Lambda$-module $R_{p^{\infty}}(E/K_{\infty})^{\dual}$. If $k$ is the order of vanishing of $f_R(T)$, then define $f^*_R(0)=T^{-k}f_R(T)|_{T=0}$ (Using part (2) of theorem 6.1 in Wuthrich's paper the value of $k$ is equal to $\corank_{\Zp}(R_{p^{\infty}}(E/K))$. It is not hard to show this latter value is 1). $f^*_R(0)$ is called the Euler characteristic of $R_{p^{\infty}}(E/K_{\infty})^{\dual}$ and its valuation is independent of the choice of $\gamma$.

To prove that $R_{p^{\infty}}(E/K_{\infty})^{\dual}$ is finitely generated over $\Zp$ we only need to show that $f^*_R(0)$ is a $p$-adic unit. We show this using part (4) of theorem 6.1 in Wuthrich. Taking into account that $\mathfrak{R}_p(E/K)$ is trivial, Wuthrich's theorem gives
\begin{equation}\label{Euler_characteristic1}
f^*_R(0) \equiv \frac{\#T_{loc} \cdot \#(R_{p^{\infty}}(E/K)/\divs)}{\#T_{gl}\cdot \#J \cdot \#I} \hspace{0.5 cm} (\text{mod} \, \Zpu)
\end{equation}

In the above $R_{p^{\infty}}(E/K)/\divs$ means the quoient of $R_{p^{\infty}}(E/K)$ by its maximal divisible subgroup. $I$ is the cokernel of the injection $\mathfrak{R}_p(E/K) \hookrightarrow T_p R_{p^{\infty}}(E/K)$ defined in lemma 3.1 of Wuthrich's paper and $J$ is the cokernel of a certain map described in his paper. Now we turn to the description of $T_{gl}$ and $T_{loc}$. In what follows let $S$ be the primes of $K$ dividing $Np$

We have $T_{gl}=H^1(\Gamma, E(K_{\infty})[p^{\infty}])$ and $T_{loc}=\prod_{v \in S} H^1(\Gamma_v, E(K_{\infty,v})[p^{\infty}])$.

In the description of $T_{loc}$ the product runs over all primes $v$ in $S$ where for every such prime $v$ we choose a prime $w$ of $K_{\infty}$ above $v$. If $K_n$ is the subfield of $K_{\infty}$ of degree $p^n$ over $K$, we denote the union of the completions of the fields $K_n$ at $w$ by $K_{\infty,v}$ and the decomposition group of $w$ by $\Gamma_v$.

We now calculate the orders of $T_{gl}$ and $T_{loc}$. When $K_{\infty}$ is the cyclotomic $\Zp$-extension of $K$, Wuthrich uses a result of Imai \cite{Imai} which states that if $E$ is an elliptic curve defined over $\Qp$ with good reduction and $L/\Qp$ is the cyclotomic $\Zp$-extension, then $E(L)_{\text{tors}}$ is finite. Using Imai's result allows one to give a nice description of $T_{gl}$ and $T_{loc}$ in the cyclotomic case. But we cannot apply Imai's result in our case so instead we will show that $\#T_{gl}=1$ and that the order of $T_{loc}$ divides the value stated in Wuthrich's paper. This will suffice for our purpose.

First we show that the fields $\Q(E[p])$ and $K$ are disjoint. To see this, note that the only primes that can ramify in $\Q(E[p])/\Q$ are primes dividing $Np$, but since $N$ was assumed to split in $K/\Q$ and $p$ not to ramify in $K/\Q$, therefore the intersection of $\Q(E[p])$ and $K$ is an unramified extension of $\Q$ and is therefore $\Q$ itself so $\Q(E[p])$ and $K$ are indeed disjoint from which it follows that $E(K_{\infty})[p^{\infty}]^{\Gamma}=E(K)[p^{\infty}]=\{0\}$ which implies that $E(K_{\infty})[p^{\infty}]=\{0\}$ since $\Gamma$ is a pro-$p$ group. So therefore we have

\begin{equation}\label{T_gl}
\#T_{gl}=1=\#E(K)[p^{\infty}]
\end{equation}\\
We now calculate the order of $T_{loc}$. To do this, we need to calculate the order of $H^1(\Gamma_v, E(K_{\infty,v})[p^{\infty}])$ for any $v \in S$. First let $v \nmid  p$ i.e. $v$ divides $N$. Since we have assumed that all the primes dividing $N$ to split in $K/\Q$, therefore by \cite{Brink} theorem 2, $v$ does not split completely in $K_{\infty}/K$ and so by \cite{CS} lemma 3.4 we have $H^1(\Gamma_v, E(K_{\infty,v})[p^{\infty}])=c_v^{(p)}$ where $c_v^{(p)}$ indicates the largest power of $p$ dividing $c_v$. Now let $v | p$ and denote the maximal divisible subgroup of $E(K_{\infty,v})[p^{\infty}]$ by $\mathcal{D}$. In this case we see from the proof of lemma 4.2 in Wuthrich's paper that we have $\#H^1(\Gamma_v, E(K_{\infty})[p^{\infty}])=\#E(K_v)[p^{\infty}]/\#\mathcal{D}^{\Gamma}$. All together, the above 2 observations imply that

\begin{equation}\label{T_loc}
t \cdot \#T_{loc}=\prod_{v | p} \#E(K_v)[p^{\infty}] \cdot \prod_{v \nmid p} c_v^{(p)}\, \,  \text{for some} \; t \in \Z
\end{equation}

By condition (a) of the theorem and the work of Kolyvagin \cite{Kolyvagin}, we have that $\Sha(E/K)$ is finite and since $\FSha(E/K)$ is a subgroup of $\Sha(E/K)[p^{\infty}]$, it follows that $\FSha(E/K)$ is finite as well. Therefore combining (\ref{Euler_characteristic1}), (\ref{T_gl}) and (\ref{T_loc}), we see from the proof of corollary 6.2 in Wuthrich's paper that for some $t \in \Z$ we have

\begin{equation}\label{Euler_characteristic2}
t \cdot \#J \cdot f^*_R(0) \equiv \#\text{Tors}_{\Zp}(D) \cdot \prod_{v \nmid p} c_v^{(p)} \cdot \#\FSha_{p^{\infty}}(E/K) \hspace{0.5 cm} (\text{mod} \, \Zpu)
\end{equation}
where $D$ is the cokernel of the localization map (induced by the maps $i_v$) from $E(K)\otimes \Zp$ to the $p$-adic completion of $\prod_{v | p} E(K_v)$

As we explained above, to prove the theorem we only need to show that $f^*_R(0)$ is a $p$-adic unit. This will follow if we can show that the right-hand side of the above equivalence (\ref{Euler_characteristic2}) is a $p$-adic unit.

From condition (b) and since we have assumed that $p$ is odd and $\Gal(\Q(E[p])/\Q)=GL_2(\Fp)$, therefore by the work of Kolyvagin \cite{Kolyvagin} $\Sha(E/K)[p^{\infty}]=\{0\}$. So we have $\FSha_{p^{\infty}}(E/K)=\{0\}$ as well since $\FSha_{p^{\infty}}(E/K)$ is a subgroup of $\Sha(E/K)[p^{\infty}]$. Also condition (d) gives $\prod_{v \nmid p} c_v^{(p)}=1$ and so the right-hand side of the equivalence (\ref{Euler_characteristic2}) is $\#\text{Tors}_{\Zp}(D)$. So we see that the theorem follows from the following proposition.
\end{proof}

\begin{proposition}
If $D$ is the cokernel of the localization map (induced from the maps $i_v$) from $E(K) \otimes \Zp$ to the $p$-adic completion of $\prod_{v |p} E(K_v)$, then under conditions (a), (c) and (e) of theorem \ref{Euler_char_theorem} $\#\text{Tors}_{\Zp}(D)=1$.
\end{proposition}
\begin{proof}

If $M$ is an abelian group we write $M^*=\ilim M/p^nM$ for its $p$-adic completion ($p$ is the prime in the proposition). Also will denote $\cap_{w |p} i_w^{-1}(E_1(K_w))$ by $E_1(K)$.

Let $\psi_p: E(K)^* \to \prod_{v |p} E(K_v)^*$ be the map in the proposition. We need to prove that $\text{Tors}_{\Zp}(\coker(\psi_p))=\{0\}$. We proceed as in \cite{CM} lemma 9. Since $\prod_{v |p} E(K_v)/E_1(K_v)$ is finite, therefore by \cite{CM} lemma 6 we have an exact sequence
$$0 \longrightarrow \prod_{v |p} E_1(K_v)^* \longrightarrow \prod_{v |p} E(K_v)^* \longrightarrow \prod_{v |p} (E(K_v)/E_1(K_v))^* \longrightarrow 0.$$
$E(K)/E_1(K)$ injects into $\prod_{v |p} E(K_v)/E_1(K_v)$ so is finite. Therefore, we also get an exact sequence
$$0 \longrightarrow E_1(K)^* \longrightarrow E(K)^* \longrightarrow (E(K)/E_1(K))^* \longrightarrow 0.$$
Condition (c) of theorem \ref{Euler_char_theorem} implies that $\prod_{v | p} (E(K_v)/E_1(K_v))^*$ is trivial. This also in turn proves that $(E(K)/E_1(K))^*$ is also trivial. So we see from the above exact sequences that we have isomorphisms $E_1(K)^* \cong E(K)^*$ and $\prod_{v |p} E_1(K_v)^* \cong \prod_{v | p} E(K_v)^*$. Therefore we see that if $\psi'_p: E_1(K)^* \to \prod_{v | p} E_1(K_v)^*$ is the map induced by $\psi_p$, then to prove the proposition we only need to show that $\text{Tors}_{\Zp}(\coker(\psi'_p))=\{0\}$.

Since $p$ is odd and unramified in $K/\Q$, therefore from \cite{Silverman} ch. 4 th. 6.4 we have that for any $v$ above $p$ the map $\log_{E,v}: E_1(K_v) \to pO_v$ is an isomorphism ($O_v$ is the ring of integers in $K_v$). Composing this isomorphism with multiplication by $p^{-1}$ we get an isomorphism $E_1(K_v) \cong O_v$. Noting that $O^*_v=O_v$, this isomorphism induces an isomorphism $E_1(K_v)^* \cong O_v$. Finally composing the map $\psi'_p$ with this last isomorphism, we get a map $\phi_p: E_1(K)^* \to \prod_{v | p} O_v$ and we will show that $\text{Tors}_{\Zp}(\coker(\phi_p))=\{0\}$. Note that the map $\phi_p$ is a $\Zp$-module homomorphism. Let $\theta: E(K) \to E(K)^*$ be the natural map. Then condition (e) of theorem \ref{Euler_char_theorem} translates to: there exists $P \in E_1(K)$ such that for some $i$ we have $\pi_i(\phi_p(\theta(P)))$ is a unit in $O_v$ ($\pi_i$ is the projection from $\prod_{v | p} O_v$ onto its $i$-th component).

Now consider 2 cases. First assume that $p$ splits in $K/\Q$. Let $v_1$ and $v_2$ be the primes of $K$ above $p$. In this case we have a map $\phi_p: E_1(K)^* \to O_{v_1} \times O_{v_2} = \Zp \times \Zp$. Condition (e) of theorem \ref{Euler_char_theorem} implies that $\pi_i \circ \phi_p: E_1(K)^* \to \Zp$ is surjective for $i=1$ or $i=2$. Without loss of generality assume that it is surjective for $i=1$. Let $(a,b) \in \Zp \times \Zp$ such that $(ra,rb) \in \img(\phi_p)$ for some $r \in \Zp \backslash \{0\}$. We must show that $(a,b) \in \img(\phi_p)$. Let $Q \in E_1(K)^*$ be such that $\phi_p(Q) = (ra,rb)$. Also since $\pi_1(\phi_p)$ is surjective there exists $P \in E_1(K)^*$ such that $\phi_p(P)=(a,c)$ for some $c \in \Zp$. We now note that $\pi_1 \circ \phi_p$ is injective. This follows from the work of Kolyvagin \cite{Kolyvagin} which shows that under condtion (a) of theorem \ref{Euler_char_theorem} $E(K)$ has rank 1. This in turn also implies that $E_1(K)$ has rank 1 since $E(K)/E_1(K)$ is finite. Then we have $\pi_1(\phi_p(rP))=r\pi_1(\phi(P))=ra=\pi_1(\phi_p(Q))$ which implies that $rP=Q$ since $\pi_1 \circ \phi_p$ is injective. Therefore $(ra,rc)=(ra,rb)$ so $b=c$ showing that $(a,b) \in \img(\phi_p)$ as desired.

Now consider the case when $p$ is inert in $K/\Q$. In this case we have a map $\phi_p: E_1(K)^* \to O_v$ where $v$ is the prime of $K$ above $p$. Note that $O_v$ is free of rank 2 over $\Zp$. Condition (e) of theorem \ref{Euler_char_theorem} implies that there exists a $P \in E_1(K)^*$ such that $Q=\phi_p(P) \in O_v^{\times}$. Then by \cite{FV} ch. 2 prop 2.4 there exists $Q' \in O_v$ such that $Q$ and $Q'$ form a basis for the free $\Zp$-module $O_v$. Since $\img(\phi_p)=\Zp Q$, therefore we easily see from this that $\text{Tors}_{\Zp}(\coker(\phi_p))=\{0\}$ as desired. This completes the proof of the proposition.
\end{proof}

The table below lists examples chosen to satisfy the conditions of theorem \ref{Euler_char_theorem}. The columns of the table are as follows: $E$ is an elliptic curve defined over $\Q$ with the given Cremona labeling \cite{Cremona}, $D$ is a fundamental discriminant such that $K=\sqrt{D}$, $p$ is a prime, $(D/p)$ is the Legendre symbol which tells us whether our unramified prime $p$ splits in $K/\Q$ ($(D/p)=1$) or is inert in $K/\Q$ ($(D/p)=-1$) and the last column is the integer $a_p=p+1-\tilde{E}(\Fp)$. Since all the entries in the table have $p \geq 5$, therefore the elliptic curves $E$ in the table with $a_p=0$ are precisely the ones with supersingular reduction at $p$. In addition to satisfying the conditions of theorem \ref{Euler_char_theorem}, the entries in the table were chosen such that $(E,p)$ satisfies $(*)$. All of the computations for the table were performed in SAGE \cite{Sage}.\\

\begin{center}
\begin{tabular}{| c | c | c | c | c |}
\hline
$E$ & $D$ & $p$ & $(D/p)$ & $a_p$ \\ \hline
11a1 & -7 & 13 & -1 & 4 \\ \hline
11a1 & -7 & 29 & 1 & 0 \\ \hline
17a1 & -8 & 11 & 1 & 0 \\ \hline
17a1 & -8 & 41 & 1 & -6 \\ \hline
43a1 & -7 & 17 & -1 & -3 \\ \hline
43a1 & -7 & 37 & 1 & 0 \\ \hline
53a1 & -11 & 19 & -1 & -5 \\ \hline
53a1 & -11 & 751 & 1 & 0 \\ \hline
57a1 & -8 & 13 & -1 & 2 \\ \hline
57a1 & -8 & 17 & 1 & -1 \\ \hline
57a1 & -8 & 37 & -1 & 0 \\ \hline
58a1 & -7 & 11 & 1 & -1 \\ \hline
58a1 & -7 & 23 & 1 & 0 \\ \hline
58a1 & -23 & 13 & 1 & 3 \\ \hline
58a1 & -23 & 139 & 1 & 0 \\ \hline
75a1 & -11 & 17 & -1 & 2 \\ \hline
75a1 & -11 & 79 & -1 & 0 \\ \hline
99a1 & -35 & 17 & 1 & 2 \\ \hline
99a1 & -35 & 71 & 1 & 0 \\ \hline
\end{tabular}
\end{center}

\end{document}